\documentclass[onecolumn]{autart} 

\usepackage{natbib}
\usepackage{amsmath,amsfonts,amssymb}
\usepackage{graphicx,subfigure}
\usepackage{enumerate}

\usepackage{float}
\usepackage{epsf}
\usepackage{fancyhdr}
\usepackage{subfigure}
\usepackage{latexsym}

\usepackage{tikz}

\usepackage{color}
\definecolor{Royalblue}{cmyk}{1,0.30,0.2,0.2}

\definecolor{greenmz}{cmyk}{0.5,0.30,0.2,0.2}

\definecolor{purple}{rgb}{0.62, 0.0, 0.77}

\newcommand{\nd}{\noindent}

\newcommand{\qq}{\qquad}

\newcommand{\diag} {\mbox{\rm diag}}
\newcommand{\rank} {\mbox{\rm rank}}

\newcommand{\Span} {\mathrm{span}\,}

\newcommand{\bmat}{\left[ \begin{matrix}}
\newcommand{\emat}{\end{matrix} \right]}

\newcommand{\beq}{\begin{equation}}
\newcommand{\eeq}{\end{equation}}
\newcommand{\tp}{^{\top}}
\newcommand{\bea}{\begin{eqnarray}}
\newcommand{\eea}{\end{eqnarray}}

\newcommand{\proof}{{\bf Proof.\ }}

\renewcommand{\theenumi}{\Alph{enumi}}

\newcommand{\E}{{\mathbb E}\,}

\newcommand{\Rbb}{\mathbb R}

\newcommand{\Zbb}{\mathbb Z}

\newcommand{\xb}{\mathbf  x}
\newcommand{\yb}{\mathbf  y}

\newcommand{\wb}{\mathbf  w}
\newcommand{\vb}{\mathbf  v}
\newcommand{\eb}{\mathbf  e}

\newcommand{\nb}{\mathbf  n}

\def\xib{\boldsymbol{\xi}}
\def\chib{\boldsymbol{\chi}}
\def\etab{\boldsymbol{\eta}}

\def\af{\mathfrak{a}}

\newtheorem{theorem}{Theorem}
\newtheorem{definition}{Definition}
\newtheorem{proposition}{Proposition}

\newtheorem{corollary}{Corollary}
\newtheorem{example}{Example}
\newtheorem{remark}{Remark}

\begin{document}

\begin{frontmatter}

\title{Hidden Factor estimation in Dynamic Generalized Factor Analysis Models}

\author[Unipd]{Giorgio  Picci}\ead{picci@dei.unipd.it},    
\author[Unipd]{Lucia Falconi \corauthref{cor}}\ead{lucia.falconi@phd.unipd.it}  
\corauth[cor]{Corresponding author.},
\author[Unipd]{Augusto Ferrante}\ead{augusto@dei.unipd.it},               
\author[Unipd]{Mattia Zorzi}\ead{zorzimat@dei.unipd.it}

\address[Unipd]{Department of Information Engineering, University of Padova, Via Gradenigo 6/B, 35131 Padova, Italy}  
          
\begin{keyword}              
Generalized Factor Analysis, Factor estimation, Kalman filtering, Kalman predictor,          
Dynamic Factor models.                                                      
\end{keyword}

\begin{abstract}                          
This paper deals with the estimation of the hidden factor  in  Dynamic Generalized Factor Analysis via a generalization of Kalman filtering.   Asymptotic consistency is discussed and it is shown that the Kalman one-step predictor is not the right tool while the pure filter yields a consistent estimate.
\end{abstract}

\end{frontmatter}

\section{Introduction} 

Factor Analysis has a long history; it has apparently first been introduced by psychologists
and  successively it has been studied and applied in various branches of Statistics and Econometrics: \cite{spearman-904,burt_1909,Ledermann-37,Ledermann-39,Bekker-L-87,Lawley-M-71}. Several dynamic versions of these models  have  also been introduced
  in the  econometric literature already a long time ago, see e.g. \cite{geweke_1977,picci_pinzoni_86,PenaBox1987,PenaJSPI2006} and
references therein.\\
With a few  exceptions, see e.g. \cite{ciccone2017factor,8976281}, little attention has been paid to Factor Analysis models in the system and control engineering community until
 recently, when we have been  witnessing a revival of interest due to  certain generalizations which bypass the problem of intrinsic unidentifiability of these models by relaxing a crucial independence condition on the noise variables. These  models, called \emph{Generalized   Factor  Analysis   (GFA)} models, although initially motivated by  econometric applications, seem to  have a potential to be quite useful also in engineering applications due to their ability of modeling time series  data of large cross-sectional dimension. The static generalization was first rigorously discussed   by   \cite{chamberlain_1983,Chamberlain-R-83}. Later, Forni, Lippi and collaborators introduced Dynamic Generalized Factor  Analysis (DFGA) models and their estimation  in a series of widely quoted papers: \cite{Forni-H-L-R-2000,forni_lippi_2001,Forni-H-L-R-02,Forni-H-L-R-03}.  Although Generalized Factor Analysis models have been primarily of interest to econometricians, this  modeling paradigm has stimulated interest   in  the   System Identification community, see \cite{Anderson-Deistler_2008,Anderson-B-D-18,Deistler-Z-07,Deistler-2010,Bottegal-P-15}. 
Indeed, modern engineering applications are characterized  by interconnected systems with hundreds or thousands of variables which are mainly ``driven'' by a  small number of  hidden factors. In plain words, generalized factor models can describe a large amount of  sensors measuring quantities which provide  information on some (few) variables of interest. 
An example of urban pollution monitoring is illustrated at the end of the paper where 
the hidden factor vector  describes the average concentration of few pollutants over time in a certain city while the observed variables   
model the measurements taken from the sensors spread in the city.
Other example from the engineering literature may be found in \cite[p. 760-761]{Bottegal-P-15}.

As this paper is mainly addressed to the System and Control community, we shall not discuss Econometric applications which are abundantly referred to in the quoted literature. Our target is the growing interest of Control Engineering researchers  in large or complex networked systems acting in the presence of uncertainty and disturbances. Modeling and identification of these systems often brings    to models with such a high dimension to involve more parameters than available data.   Factor Analysis is believed to provide an answer to this issue as it  concentrates the explication of high dimensional data or observables in a (hopefully) small-dimensional vector of  latent factors. 
State space analysis of DGFA models seems to offer the ideal approach to this end. Indeed, 
it provides natural tools to estimate the hidden factor variable of a DGFA model which is the important problem addressed in this paper. Classical state space estimation techniques originated with  the work of R.E. Kalman seem to be the perfect device to tackle the problem. Some other works in  recent literature address factor estimation by  state-space  techniques, see e.g.  \cite{Kapetanios-M-09,Marcellino-17},
\cite{Doz-Giannone-Reichlin-11},  \cite{Doz-Giannone-Reichlin-11} and \cite{GIANNONE2008}.
We refer to \cite{Stock-W-011,Stock-W-015} for a survey mentioning some of the themes treated in these papers, especially the important problem of dynamic  factor estimation which will be the main concern of this paper. \\
For simplicity, we make the same   stationarity  assumptions as in these papers, although the main arguments of our paper can be shown to hold also if we replace stationarity by some
form of weak time dependence. The main difference in the setup of this paper  with respect to \cite{Doz-Giannone-Reichlin-11} is that we  discuss  the Kalman filter and the one-step ahead predictor for DGFA models instead of the  fixed-interval smoother.
The essential ideas behind the motivation and structure of Generalized Factor Analysis are shortly reviewed in the next two sections.

\subsection{Contributions of the paper}

We consider a class of state-space DGFA models and analyze the capability of Kalman estimators (filter and predictor) to provide a {\em perfect asymptotic estimate}
namely estimates for which the covariance of the estimation error tends to zero when the number of observed outputs diverges.
The main contributions of this paper are twofold: 
\begin{enumerate}
\item
We show that the one-step ahead Kalman predictor does not provide
a perfect asymptotic estimate of the hidden factors.
As a consequence, we show that a generalized dynamic Factor Analysis model
cannot be a predictor-based innovation model.

\item
We prove that the pure filter, on the contrary, provides
a perfect asymptotic estimate of the latent state variable.
Moreover, under reasonable assumptions, the estimation error converges weakly to the
idiosyncratic noise generating the data so that generalized dynamic Factor Analysis models
are weakly equivalent to a pure-filter type innovation models.
\end{enumerate}
Besides this two main contributions we derive some ancillary results on static
Factor Analysis models that are interesting, {\em per se}. \\
Notice that, by combining our two main results, we come to  a significant warning concerning
the application of  standard subspace identification methods to the identification of GDFA models.
In fact, subspace  methods are normally based on the  one-step ahead Kalman predictor model which we show  cannot produce a generative model exhibiting   GDFA features.
On the other hand, considering the pure filter estimate
in place of the predictor can overcome this problem.

\nd{\em Notations:} In this paper boldface symbols  will normally denote
random variables or random arrays, either finite or infinite. All random variables will be real,  zero-mean and with finite variance.  The symbol $H(\vb)$ denotes the standard
Hilbert space of random variables linearly generated by the
scalar components $\{\vb_1,\ldots,\vb_n,\ldots \}$ of a (possibly infinite)
family of random variables which we denote by $\vb$. For $\xib,\,\etab \in H(\vb)$, the inner product  is the mathematical expectation $\langle \xib,\,\etab\rangle:= \E [\xib \etab]$ which induces  the (variance) norm of  random variables by setting $\|\xib\|^2 = \E [\xib^2]$.   Convergence of random sequences will always be understood with  respect to this norm. Finally, $\mathrm{diag}\{a_1,\ldots a_n\}$ denotes the diagonal matrix whose elements in the main diagonal are $a_1\ldots a_n$. \\
This paper is concerned with 
{\em model-based} estimation,  in particular, with  the problem of the model-state estimation  by  a Kalman one-step ahead predictor, or by a Kalman filter. We highlight this point because in the econometric literature   
the term ``estimation" normally refers to a parameter estimation procedure for estimating the model and consequently the involved time-variables on the sole basis of the observed data.  

\section{Review of static Factor Analysis models}\label{Sec:FA}

A classical (static) {\it Factor Analysis\/} model is a representation of the form
$$
\yb=F\xb+\eb,
$$
of $N$ observable random variables $\yb=[\,\yb(1)\,\ldots\,\yb(N)\,]^{\top}$,
as linear combinations
of $q$ random {\it common factors\/} $\xb=[\,\xb_1\,\ldots\,\xb_q\,]^{\top}$,
plus uncorrelated ``noise" or ``error" terms
$\eb=[\,\eb(1)\,\ldots\,\eb(N)\,]^{\top}$. The columns $\{f_1,\;f_2,\;\ldots,f_{q}\}$ of matrix
$F$, called the {\it factor loadings}, can be chosen to be linearly independent. Moreover, the common factors can be  normalized in such a way  that $  \E [\xb\xb^{\top} ]= I$.    An essential part of the model
specification is  that the  $N$ components of the error $\eb$ should
be (zero-mean and) {\em mutually uncorrelated} random variables, i.e.
$$
 \E  [\xb\eb^{\top}]=0\,,  \qquad \E
[\eb\eb^{\top}]=\diag\{\sigma^{2}_1,\ldots,\sigma^{2}_N\}\,.
$$
The aim of these models is to provide an ``explanation" of the
mutual correlations of the observable variables $\yb(i)$ in terms
of a small number $q$ of common factors, in the sense that, setting: $ \hat{\yb}(k):=\sum f_i(k)\xb_i$, where $f_{i}(k)$ is the $k$-th component of $f_i$, one has exactly   $\E [\yb(i)\yb(j)]=\E[ \hat{\yb}(i)\hat{\yb}(j)]$,
for all $i\neq j$. Note that  a Factor Analysis representation then induces a decomposition of the covariance matrix $\Sigma$ of $\yb$ as
\begin{equation}
\label{SpL}\Sigma= F F^{\top} + \diag\{\sigma_{1}^2,\ldots,\sigma_{N}^2\}
\end{equation}
which can be seen as a special kind of \emph{low rank  plus sparse} decomposition of a covariance matrix, see \cite{Chandra-etal-011}, a diagonal covariance  matrix  being as sparse as one could possibly ask for.

Although providing a quite natural and
useful data compression scheme, factor models in many circumstances suffer from a
serious non-uniqueness problem coming from the fact that, even for a fixed dimension $q$ there are in
general many (possibly infinitely many) statistically non-equivalent Factor Analysis models describing the same  family of observables $\yb(1),\ldots,\yb(N)$. In addition, determining the minimal integer $q$ for which a decomposition as in (\ref{SpL}) holds for a given symmetric positive definite matrix $\Sigma$ has been an open problem since the beginning of the last century.  Moreover, there are in general
many  minimal Factor Analysis models (say with $F$'s of the same rank $q$ and   normalized factors)
representing a fixed $N$-tuple of random variables $\yb$. 
This inherent nonuniqueness of Factor Analysis models has been   called ``factor indeterminacy"
and corresponds to  {\em unidentifiability} in the systems and control language.
The overlooking of this unidentifiability issue and the acritical usage of Factor Analysis models 
have been  vehemently criticized by Kalman in a series of papers, see e.g.  \cite{Kalman-83,KalmanDFM1983}.

One may try to get uniqueness by giving up or mitigating the requirement of uncorrelatedness of the components of  $\eb$. Obviously this tends to make the problem  ill-defined  as the basic goal  of uniquely splitting the external signal into a noiseless component plus ``additive noise'' is made vacuous, unless some  extra assumptions are made on the model and on the very notion of ``noise''.   Quite surprisingly, for  models describing an {\em infinite} number of observables a meaningful weakening of the property of uncorrelation of the components of $\eb$ can be introduced, so as to guarantee  the uniqueness of the decomposition.

\section{Generalized Factor Analysis models}\label{Idiosync}

The covariance matrix of an  infinite-dimensional  zero mean  vector $\yb=[\yb(k)]_{k=1,2,\ldots}$ is
formally written as  $\Sigma := \mathbb{E}  \yb \yb^{\top} $. We let
$\Sigma_N$ indicate  the top-left $N \times N$ block of $\Sigma$,
equal to the covariance matrix of the subvector made of the  first $N$ components of $\yb$ denoted    $\yb_N$.
The inequality $\Sigma > 0$  means that all submatrices $\Sigma_N$
of $\Sigma$  are   positive  definite, which we shall always assume
in the following.

Let $\ell^2(\Sigma)$ denote the Hilbert space of infinite sequences
$a:= \{a(k),\,k \in \mathbb{N}\}$ such that $\| a \|^2_\Sigma :=
a^{\top} \Sigma a < \infty$. When $\Sigma = I$, we
simply use the symbol $\ell^2$ and denote the corresponding norm with the symbol $\|\cdot\|$.

A side question discussed in the appendix is the relation between $\ell^2$ and $\ell^2(\Sigma)$: indeed,
for $N$ finite it is obvious that $\ell^2=\ell^2(\Sigma)$ since $\lambda_{\rm max}I_N >\Sigma > \lambda_{\rm min}I_N$, $\lambda_{\rm max}$ and $\lambda_{\rm min}$ being the maximum and minimum eigenvalue
of $\Sigma$, respectively.

\medskip

\begin{definition}[\cite{forni_lippi_2001}]
  Let $\af:=\{a_n,\, n \in \mathbb{N}\}$ be a sequence of elements
of  the space $\ell^2 \cap \ell^2(\Sigma)$. We say that $\af=\{a_n,\; n
\in \mathbb{N}\}$ is an {\bf averaging sequence} (AS) if
$$\lim_{n\rightarrow\infty} \| a_n \| = 0.$$
\end{definition}
\begin{example}\label{example1}
{\em The sequence of elements in $\ell^2$
$$
a_n = \frac{1}{n}  [\,\underbrace{1 \, \ldots \,1}_{n}\,0\,\ldots\,
]^{\top}
$$
is an averaging sequence.}
\end{example}
An AS can be seen just as a sequence of linear functionals in
$\ell^2 \cap \ell^2(\Sigma)$ converging strongly to zero.  The
definition   is instrumental  to   the concept of {\bf idiosyncratic
sequence} of random variables which will be introduced next.
\begin{definition}[\cite{forni_lippi_2001}]
  We say that the random sequence $\yb$ is {\bf idiosyncratic} if   
 $ \; \lim_{n\rightarrow\infty} a_n^{\top} \yb = 0$  for any averaging
sequence $\af=\{\, a_n \in \ell^2 \cap \ell^2(\Sigma)\}$. The limit is understood in mean square.
\end{definition}

\begin{remark} {\em  In econometrics, ``idiosyncratic error'' is used to describe unobserved factors that impact the dependent variable. 
The  mathematical definition  of idiosyncratic disturbance was introduced in the static GFA case by  \cite{chamberlain_1983},
and \cite{Chamberlain-R-83}. Since then, models with infinite cross-sectional size and  idiosyncratic errors    have been analyzed by many authors. Beside
\cite{forni_lippi_2001}, that we take as our reference for the definition of idiosyncratic sequence, we mention \cite{Forni-Reichlin-96}, \cite{Stock-W}, \cite{Forni-Reichlin-98}, \cite{Forni-etal-2000}, \cite{Forni-H-L-R-02}, and \cite{Forni-H-L-R-03}.
We also mention that in \cite{forni_lippi_2001}, in place of averaging sequences of elements in $\ell^2$, a more general concept of dynamic averaging sequences is introduced. This concept seems to be useful  only when the observable $\yb$ is described in the spectral domain by spectral densities   in place of covariances.}
\end{remark}

\begin{example}
{\em A zero-mean sequence whose variance is a bounded operator in $\ell^2$ is idiosyncratic. In fact let the operator norm $\|\Sigma\|$ be bounded by $\alpha >0$. Then $\Sigma \leq \alpha I$ where $I$ is the identity operator, so that,
$$
\E [(a_n^{\top}\yb)^2]= a_n^{\top} \Sigma a_n \leq \alpha  \|a_n\|^2 \rightarrow 0
$$
for any  sequence $\{a_n\}$ tending to zero in norm.}
\end{example}
The characterization   actually goes both ways:
\begin{proposition}
A zero-mean sequence of random variables is idiosyncratic if and only  if its  variance matrix is a bounded operator in $\ell^2$. 
\end{proposition}
It follows, as remarked in   \cite{forni_lippi_2001}, that a sequence $\yb$ is  idiosyncratic if and only if 
there exists $M$ such that $\forall N>0$, $\|\Sigma_N\|\leq M$,
where  $\Sigma_N$ is the covariance of the vector $\yb_N$
obtained by selecting the first $N$ components of $\yb$
(see \cite{Bottegal-P-15} for more details).
In particular, uncorrelated (white) sequences of zero-mean random variables having a uniformly  bounded  variance are idiosyncratic.

Consider a  zero-mean   finite variance   stochastic process $\yb :=\{\yb(k),\,k \in \Zbb_+\}$ represented  as a random column vector with an infinite number of components. 
\begin{definition}\label{def:AggrIdio}
A   {\em Generalized Factor Analysis Model (GFA)} of the process $\yb$ is  a representation by a finite  linear combination  of unobserved  random
variables plus  {\em idiosyncratic noise}, of the form 
\begin{equation}\label{GFAD}
\yb(k) = \sum_{i=1}^{q}f_i(k) \xb_i + \tilde{\yb}(k)\,,\qq k=1,2, \ldots
\end{equation}
where the random variables $\xb_i\,,\, i=1,\ldots,q$, are called the {\em common factors} and the deterministic infinite-dimensional real vectors $f_i$, called the {\em factor loadings},  are {\em strongly linearly independent} (the definition being reviewed in the appendix).  
The   components $\tilde{\yb}(k)$'s of the idiosyncratic noise are zero mean  random variables   orthogonal to $\xb$. 
\end{definition}

We  order the linear combinations $ \hat{\yb}(k):=\sum f_i(k)\xb_i$, with $k=1,2,\ldots$,   into an  infinite random vector $\hat{\yb}$ and likewise for the noise terms $\tilde{\yb}(k)$ so that \eqref{GFAD} can for short be written $\yb =\hat {\yb} + \tilde{\yb}$ where the {\em hidden or latent} component $\hat {\yb}$ is a linear function of $\xb$ e.g. $\hat {\yb} =F \xb$. Since the columns of $F$ are linearly  independent, $H(\hat {\yb})=H(\xb)$ and the common factors are just obtained by picking an orthonormal basis in  $H(\hat {\yb})$.

The key characteristics  which qualifies the idiosyncratic noise process $\tilde{\yb}$    is that by a deterministic linear averaging operation on the components of the observation vector $\yb$,  one can (asymptotically) eliminate the additive noise and reveal the common  factors.  The question is how to construct a suitable AS from the  model specifications to achieve this goal. This  will be one of the the main themes of the next  sections.  

\subsection{Application and interpretation of GFA models}
We now suggest an interpretation of GFA models in the framework of  applications 
to ensembles of a large number of agents   distributed in space and
interacting in a random fashion. We consider two distinct mechanisms:
\begin{enumerate}
\item {\bf Short distance interaction.} The idiosyncratic
covariances $ \{\tilde\sigma(k,j)\}$ describe  the mutual influence
of neighbouring  units having coordinates $\tilde\yb(k),\,\tilde\yb(j)$. Let $\tilde \Sigma$ denote the covariance matrix of $\tilde\yb$. Since $
\tilde{\Sigma}$ is a bounded operator in $\ell^2$, it is a known
fact \cite[Section 26]{Akhiezer-G-61} that  $
\tilde\sigma(k,j)\rightarrow 0$ as $|k-j|\rightarrow \infty$ so in a
sense the idiosyncratic component $\tilde\yb$ models only {\em short
range interactions} among the agents, which are decaying with
distance. Agents which are far away from each other are not affected by
mutual influence.
\item {\bf Factor loadings and long range influence.}
Since $\E [\hat{\yb}(k)\hat{\yb}(j)]= \sum_{i=1}^n f_i(k)f_i(j)$ and  the
elements of the column vectors  $f_i  \in \Rbb^{\infty}$ do not
decay with distance, the products $f_i (k)f_i(j)$ do not vanish when
$|k-j|\rightarrow \infty$. Hence     the factor loadings describe
``long range'' correlation between the factor components and the
$\hat{\yb}$ component of $\yb$ can be interpreted as variables
produced by  {\em long range interaction} among agents.
\end{enumerate}
In dynamic GFA models the components of the infinite vector process $\hat\yb(t)$ move in time like a rigidly connected set of points in space. For this reason $\{\hat\yb(t)\}$ is called the {\em flocking component} of the process $\{\yb(t )\}$ \cite{Bottegal-P-15}.

\section{Dynamic State-Space DGFA models}
Let us  introduce the time variable $t\in \Zbb$ and denote  now  by\footnote{ We warn the reader that in the dynamic setting the notation $\yb(t)$ denotes an infinite random vector and must not be confused with the notation of the static setting where $\yb(k)$  denotes a scalar random variable.}
$\yb:= \{\yb(t),\, t\in \Zbb\}$ a zero-mean, stationary, vector process of infinite cross-sectional dimension so that at each time $t$ the random vector $\yb(t)$ has countably infinite  random components. 
We consider the following  finite-dimensional\footnote{ Of course, more general 
infinite-dimensional DGFA models are possible that are not included in our systems' class.} dynamic model  where $\yb(t)$ depends on  a vector of $n$  common factors, $\xb$, evolving according to a linear dynamics of the form \footnote{For consistency with standard system-engineering notations (see e.g. \cite{LPBook}), we have denoted the  factor loading (also called output or observation) matrix by $C$ instead of $F$ as in the Factor Analysis literature.
}
\begin{equation}\label{Modclass-inf}
\left\{
\begin{array}{l}\xb(t+1) = A \xb(t) +  \vb(t) \\
\yb(t) = C \xb(t) +  \wb(t) \,.
\end{array}
\right.
\end{equation}

The study of GDFA's \eqref{Modclass-inf} will be undertaken by considering sequences of truncated models   of  increasing cross-sectional dimension $N$, each describing  the subvector   $\yb_N$  made of the first $N$ components of   the original output vector $\yb$. 
More precisely, we consider  a class
of truncated models of the form
\begin{equation}\label{Modclass}
	\left\{
	\begin{array}{l}\xb(t+1) = A \xb(t) +  \vb(t) \\
		\yb_N(t) = C_N \xb(t) +  \wb_N(t) \,,
	\end{array}
\right.
\end{equation}	
where  the state dimension $n$ of the  model  
	(\ref{Modclass-inf}) is fixed  (and therefore does not grow
	with the output dimension $N$). Each output matrix  
	$ C_{N} \in \, \Rbb^{N\times n}$ is the top submatrix of $C$ of dimension $N\times n$  so that  $C_N$ 
	has the nested structure
	$$
	C_{N +k}= \bmat C_N\tp & \; \tilde C_k\tp \emat \tp \,,
	$$
	and the noise vectors $\wb_N(t)$ have a similar nested structure. Notice that, by these assumptions  each $\yb_N(t)$ is a N-vector stationary process.
	
	We assume that
\begin{enumerate}
	\item\label{primaass}
	The  $n$-dimensional latent factor $\xb\equiv \{\xb(t); t\in \Zbb\}$   follows a stationary Markov  evolution
	described by the first of equation in (\ref{Modclass-inf}), where $A\in \Rbb^{n\times n}$ is  an asymptotically stable matrix (all its eigenvalues have modulus strictly less than one) and   $\vb(t)$  a white noise process of dimension $n$ whose covariance is denoted by $Q$.
	Hence  the steady-state variance of  $\xb(t)$ is the unique solution of the Stein (discrete-time Lyapunov) equation $\Sigma =A\Sigma A\tp+ Q$. 
	
	\item 
	The infinite-dimensional  white noise vector $\{\wb(t)\}$ is not assumed to have uncorrelated components as in standard Factor Analysis models  and (without loss of generality) is assumed to be uncorrelated with $\{\vb(t)\}$  at all times.
\end{enumerate}

	The following  assumptions regarding the asymptotic behaviour for $N\to \infty$ of  the sequence of models \eqref{Modclass} will be made: 
	\begin{enumerate}\setcounter{enumi}{2}
	\item \label{Ass2}
	$C$ is an $\infty \times n$  matrix with {\em strongly linearly independent columns}. 
	This  means that 
	\beq\label{slic}
	\lim_{N\rightarrow\infty}\lambda_{\min}[C_N\tp C_N]=+\infty,
	\eeq
	where $\lambda_{\min}[\cdot]$ denotes the smallest eigenvalue, see the appendix for a discussion and for more details on this.	
	
	\item \label{Ass3} The noise $\wb(0)$ is an idiosyncratic sequence (with respect to the cross sectional dimension).
	\end{enumerate}

	These conditions can easily be shown to be equivalent to conditions (7) and (8) in \cite{Stock-W-011}. 
	Since $\wb_N(t)$ is a stationary white noise process (in~$t$) for all $N$,
	it is easy to check that Assumption \ref{Ass3} implies that for any given $t$, $\wb(t)$ is also idiosyncratic.\\
	To avoid technicalities, we also assume that 
	\begin{enumerate} \setcounter{enumi}{4}
		\item \label{Ass4} The model is wide-sense stationary and minimal  so that  the pair $(A,Q)$, where $Q$ is the covariance of the white noise $\vb$, is reachable.
		
		\item \label{Ass5} The covariance $R_N$ of the output noise $\wb_N(t)$ is positive definite, i.e. $R_N>0,\,\forall N$.
	\end{enumerate}

 As observed by \cite{Deistler-2010}, minimality of the model (\ref{Modclass-inf}) does not guarantee that the state is a  {\em minimal static factor} of the observed process.
 
 \section{The Static case: latent variable estimation}\label{StochReal}

Are these models realistic and well posed and if so, how can we use them to compute estimates of  the latent vector $\xb$ based on observations of $\yb$ (and knowledge of the parameters)? In order to  get a hint on how to answer  this basic question we shall  look  at the following simple but  illuminating example.

\begin{example}[Estimation by averaging]\label{Ex1}
{\em Let ${\mathbf 1}\!\!1$ be an infinite column  vector of $1$'s and
let $\xb $ be a scalar random  variable uncorrelated with  $
\wb$, an infinite-dimensional vector whose components form a zero-mean  weakly stationary    sequence having finite variance.
Consider the static GFA model
\begin{equation*}
\yb = {\mathbf 1}\!\!1 \xb + \wb\,.
\end{equation*}
 It is easy to show that    $\lim_{n\rightarrow\infty} a_n^{\top} 
\wb = 0$ for any averaging sequence so that $\wb$ is idiosyncratic. In particular, for the AS of Example 1, we have $L^2-\lim_{n \to \infty}\, \frac{1}{n}\, \sum_{k=1}^n \yb(k) = \xb\,
$  and  we can recover the latent factor  by averaging.} \hfill$\Box$
\end{example}
In this example we can  asymptotically   estimate {\em without error} the factor by averaging. This method exploits the two basic properties of  GFA models, in particular idiosyncrasy of the noise, by passing to the limit in $N$. But how general  is this procedure? Below we shall consider a general static model and  estimate $\xb$  by the standard linear Bayes  rule. The idea  is to show that Bayes rule  provides indeed a way to construct averaging sequences.\\
 Consider the sequence  of   static finite-dimensional linear models 
\beq\label{staticGFA}
\yb_N = C_N \xb +  \wb_N \, \qquad N=1,2,3, \ldots
\eeq 
where  $\xb$  is a fixed random vector (not depending on $N$) and $\wb_N$ are zero-mean  random vectors uncorrelated with $\xb$. Without loss of generality  the variance matrix of $\xb$ is normalized to the identity. The question is how to  estimate (or reconstruct) the hidden variable $\xb$ in the representation (\ref{staticGFA})  based on an infinite string of observations.\\
Let $\tilde{\Sigma}_N:=\E[ \wb_N  \wb_N\tp]$ so that, 
$$
\Sigma_N= C_NC_N^{\top}+ \tilde{\Sigma}_N
$$
\begin{proposition}\label{prop2}
Consider  a class
of truncated models   of the form  \eqref{staticGFA}, where the sequences $C_N$ and $\tilde{\Sigma}_N$ are given, and assume that \eqref{slic}   holds and that $\wb_N$ converges to an idiosyncratic process. Then,  for $N\to \infty$,
the optimal Bayes estimate $\hat \xb_N := \hat{\E}[{\xb} \mid \yb_N]$ converges in mean square to the actual vector $\xb$ which can therefore be reconstructed with arbitrary precision provided that 
$N$ is sufficiently large.
\end{proposition}
\proof
The classical orthogonal projection formula yields
\begin{align*}
\hat \xb_N := \hat{\E}[{\xb} \mid \yb_N] &= 
(I+C_N^{\top}\tilde{\Sigma}_N^{-1} C_N)^{-1}\,
C_N^{\top}\tilde{\Sigma}_N^{-1}\,\yb_N \\
&=
C\tp_N(C_NC^{\top}_N+\tilde{\Sigma}_N)^{-1}\,\yb_N.
\end{align*}
The covariance of the estimation error ${\boldsymbol{\varepsilon}}_N:=\xb-\hat\xb_N$  is given by
\beq
\E[{\boldsymbol{\varepsilon}}_N{\boldsymbol{\varepsilon}}_N\tp]=(I+C_N^{\top}\tilde{\Sigma}_N^{-1} C_N)^{-1}
\eeq
because the estimation error ${\boldsymbol{\varepsilon}}_N$ and the estimate $\hat\xb_N$ are orthogonal.
Since $\tilde{\Sigma}_N\leq \alpha I_N,$ we have $\tilde{\Sigma}_N^{-1}\geq \alpha^{-1} I_N,$ so that
\beq\label{eq12}
\E[{\boldsymbol{\varepsilon}}_N{\boldsymbol{\varepsilon}}_N\tp]\leq (I+\alpha^{-1}C_N^{\top} C_N)^{-1}\leq \alpha(C_N^{\top} C_N)^{-1}
\eeq
and since the columns of $C_N$ are strongly linearly independent, Equation (\ref{slic}) holds so that
$$
\lim_{N\rightarrow\infty}\E[{\boldsymbol{\varepsilon}}_N{\boldsymbol{\varepsilon}}_N\tp]=0.
$$
In other words, the estimation error of the hidden variables converges to zero in mean square.
\qed

Notice that for simplicity, we have assumed that the sequence  $\tilde{\Sigma}_N$ is known but,
with an argument similar to that in \cite{Doz-Giannone-Reichlin-11} we may 
show that the same result holds also when this is not the case.

Since the  finite-data innovation (i.e., the estimation error $ \eb_N:= \yb_N-C_N\hat{\xb}_N$)  is orthogonal to $\hat{\xb}_N$, one also has the orthogonal (innovation) representation
$$
\yb_N=C_N\hat{\xb}_N+ \eb_N
$$
and, since under the stated assumptions on the model \eqref{staticGFA}, $\hat\xb_N \to  \xb$ as $N\to \infty$,  
we have the following result.
 \begin{proposition}
Under the assumptions of Proposition \ref{prop2} and assuming also that $\tilde{\Sigma}_N$ is uniformly coercive\footnote{This means that there exists $c>0$ independent of $N$ such that $\tilde{\Sigma}_N\geq c I_N$. This assumption is only made to avoid technical issues and might be weakened.} and that  $C_N$ is uniformly bounded, \footnote{This means  that $\exists \beta$ independent of $N$ such that $\max_{i,j}|[C_N]_{ij}|\leq \beta$, where 
$[C_N]_{ij}$ is the entry in row $i$ and column $j$ of $C_N$.}
any model \eqref{staticGFA} tends, for $N\to \infty$,  to become an innovation model. In particular, 
each entry of the finite-data innovation 
$\eb_N$ tends, in mean-square sense, to the corresponding entry of $\wb_N$ and the limit innovation process tends to be idiosyncratic.
\end{proposition} 
 \begin{proof}
Let ${\boldsymbol{\delta}}_N:=\eb_N-\wb_N=C_N{\boldsymbol{\varepsilon}}_N$. We have
\begin{align*}
\E[{\boldsymbol{\delta}}_N{\boldsymbol{\delta}}_N\tp] &=
C_N(I+C_N^{\top}\tilde{\Sigma}_N^{-1} C_N)^{-1}C_N\tp \\
&\leq 
\alpha C_N(C_N^{\top} C_N)^{-1}C_N\tp
\end{align*}
where the last inequality follows from the same argument that led to (\ref{eq12}).
Let $i\in\{1,2,\dots, N\}$ and consider the $i$-th component
$[{\boldsymbol{\delta}}_N]_i$ of ${\boldsymbol{\delta}}_N$.
Its covariance $\E[([{\boldsymbol{\delta}}_N]_i)^2]$ satisfies
$$
\E[([{\boldsymbol{\delta}}_N]_i)^2]\leq \frac{\alpha}{\sigma_{min}[C_N^{\top} C_N] } [C_N]_i  [C_N]_i^{\top},
$$
where $[C_N]_i$ is the $i$-th row of $C_N$.
Since by assumption $C_N$ is uniformly bounded, $ [C_N]_i  [C_N]_i^{\top}$ is also  uniformly bounded so that there exists
$\gamma$ independent of $N$ such that 
$$
\max_{i}\E[([{\boldsymbol{\delta}}_N]_i)^2]\leq \frac{\alpha\gamma}{\sigma_{min}[C_N^{\top} C_N] }\stackrel{N\rightarrow\infty}{\longrightarrow}0
$$
because $\sigma_{min}[C_N^{\top} C_N]$ diverges in view of the strongly linear independence assumption.
To see that $\eb_N$   tends to be idiosyncratic, notice that $\eb_N={\boldsymbol{\delta}}_N+\wb_N$
so that
\begin{align*}
\E[\eb_N \eb_N\tp] 
&\leq 2\tilde{\Sigma}_N +2 \E[{\boldsymbol{\delta}}_N{\boldsymbol{\delta}}_N\tp]\\
&\leq 2\alpha (I_N+C_N(C_N^{\top} C_N)^{-1}C_N\tp)\\
&\leq 4\alpha I_N
\end{align*}
which shows that the covariance  $\E[\eb_N \eb_N\tp]$ converges monotonically to a bounded operator in $\ell^2$. 
\qed
\end{proof}

{\bf Remarks}\\
1. The mean-square entry-wise convergence of $\eb_N$ to $\wb_N$
is a weak form of convergence that may be interpreted as convergence in a mixed mean-square and infinity-type norm. 
It is natural to ask whether a stronger convergence holds in which the entire vector $\eb_N$  converges to $\wb_N$ or, equivalently whether  ${\boldsymbol{\delta}}_N$ does  converge to zero in mean square.
This indeed is not the case. In fact, 
for $N$ sufficiently large, 
$$
I+C_N^{\top}\tilde{\Sigma}_N^{-1} C_N\leq I+c^{-1}C_N^{\top} C_N\leq 2c^{-1}C_N^{\top} C_N,
$$
where the last inequality depends, in view of (\ref{slic}),  on the fact that the smallest singular value of $C_N^{\top} C_N$ diverges.
Then, 
$$
\alpha C_N(C_N^{\top} C_N)^{-1}C_N\tp\geq \E[{\boldsymbol{\delta}}_N{\boldsymbol{\delta}}_N\tp]\geq
\frac{c}{2} C_N(C_N^{\top} C_N)^{-1}C_N\tp.
$$ 
Finally notice that $C_N(C_N^{\top} C_N)^{-1}C_N\tp$ is the matrix orthogonally projecting on the image of $C_N$, so that it has rank equal to $n$ and $n$ eigenvalues equal to $1$.
In conclusion, for any (large enough) $N$ the covariance matrix of 
${\boldsymbol{\delta}}_N$ is a matrix of rank $n$ whose $n$ largest eigenvalues are bounded above by $\alpha$ and are greater than $c/2$.\footnote{At the price of some technical complication in the derivation we could refine the lower bound to $c/(1+\epsilon)$ for any positive $\epsilon$.}
Therefore, $\eb_N$ does {\em not} tend to $\wb_N$ in mean square.\\
2. Notice that the assumption of Proposition \ref{prop2} that $C_N$ is uniformly bounded can be weakened but cannot be completely eliminated. In fact, it is easy to see that the result still holds, for example, for $C_N=[1\ 2\ \dots\ N]\tp$ but does not hold, for example, for $C_N=[1\ 2\ \dots\ 2^{N-1}]\tp$.
Indeed, in the former case,  
$\max_{i}\E[([{\boldsymbol{\delta}}_N]_i)^2]
\leq\frac{6N^2\alpha}{N(N+1)(2N+1)}\stackrel{N\rightarrow\infty}{\longrightarrow}0$, but in the latter case
$\max_{i}\E[([{\boldsymbol{\delta}}_N]_i)^2]
\geq\frac{3c 4^{N-1}}{2(4^N-1)}\stackrel{N\rightarrow\infty}{\longrightarrow}3c/8>0$.
\\
3. In general, the error process  $\eb_N$ does not need to have uncorrelated components and hence  this procedure can in principle be applied to GFA models and extended to   the dynamic case. The key question however is to check if $\eb_N$ {\em can converge  to an idiosyncratic sequence.}\\
In the static case the answer to this question is
affirmative as shown before. Interestingly and somehow surprisingly, this is not guaranteed to happen in the dynamic case.

\section{The Dynamic case: the Kalman Predictor}
Asymptotic factor  (actually factor space) estimation by linear Bayesian averaging
 can be generalized to Dynamic GFA models. This  is also implicitly hinted at in  \cite[p. 9]{Stock-W-011}.  For each finite $N$ one can estimate the latent variable $\xb(t)$ in the GDFA model \eqref{Modclass} by Kalman filtering, see  e.g. \cite{Kapetanios-M-09}.  The usual understanding of Kalman filtering  leads to compute  the one-step ahead estimate $\hat \xb_N(t | t-1)$ of the hidden variable $\xb(t)$ based on previous outputs up to time $t-1$. The estimator can obviously be implemented for finite truncations of the model, of increasing dimension $N$  and, assuming steady state, leads to the following sequence  of innovation models:
\beq \label{Predittore}
\left\{
\begin{array}{l}
\hat \xb_N(t+1|t)= A \hat \xb_N(t|t-1) + K_N \eb_N(t)\\    
\yb_N(t) = C_N\hat \xb_N(t|t-1) + \eb_N(t),              
\end{array}
\right.
\eeq
where the innovation
$$
\eb_N(t):=\yb_N(t)-C_N\hat{\xb}_N(t|t-1)
$$
has covariance
\beq\label{lambdan}
\Lambda_N=C_NP_NC_N^{\top}+R_N
\eeq
with $P_N$ being the  stabilizing solution of the Algebraic Riccati Equation (ARE)
\beq\label{ARE}
P_N=A[\,P_N-P_NC_N^{\top}(C_NP_NC_N^{\top} +R_N)^{-1}C_NP_N\,]A^{\top}+Q
\eeq
from which  one can compute  the {\em Kalman gain}, 
$$
K_N:=AP_NC_N^{\top}\Lambda_N^{-1}\,.
$$
The question is how the estimates behave for $N\to \infty$. In particular  the question is if for $N\to \infty$ the (stationary) innovation  representation  (\ref{Predittore}) of $\yb_N(t)$ is  a legitimate DGFA representation. Given the standing assumption on $C$, this will be  true   if and  only if  the limit innovation process  $\eb_N(t) $ is  idiosyncratic. Now the (steady-state) innovation covariance is given by (\ref{lambdan}) and since the original model is assumed  DGFA, by Assumption (\ref{Ass3})    the output noise covariance  $R_N:=\E [ \wb_N(t)\wb_N(t)^\top]$ tends for $N\to \infty$ to a bounded covariance operator $R$. The term $C_NP_NC_N^{\top}$ can be interpreted as a perturbation of $R_N$ and it looks like   such a perturbation tends to be {\em unbounded}. In fact, under the minimality (and hence by the reachability) assumption, for each finite $N$, $P_N>0$, and since $C_N$ has strongly linearly independent columns one may be led to conjecture that  $C_NP_NC_N^{\top}$ converges to an unbounded operator of finite rank $n$. This fact casts doubts on the model \eqref{Predittore} converging to   a legitimate DGFA model for $N\to \infty$. The argument can be made rigorous as stated in the following theorem.
\begin{theorem}\label{th:kalman_pred}
Consider  a class of truncated models of the form \eqref{Modclass} with Assumptions (\ref{primaass}) to (\ref{Ass5}).
For $N\to \infty$ the innovation process $\eb_N$ in  the steady-state innovation representation \eqref{Predittore}
 does not tend to an idiosyncratic process. Hence for $N\to \infty$ the innovation model \eqref{Predittore} does not tend to  a legitimate DGFA model.
\end{theorem}
\proof  
By minimality of the model \eqref{Modclass}, $(A,Q)$ is reachable and hence the stabilizing solution $P_N$ of the ARE \eqref{ARE} is positive definite for each fixed $N$.
Then we can rewrite the ARE \eqref{ARE} as
\beq\label{ARE1}
P_N=A\left( \,P_N^{-1}+C_N^{\top}R_N^{-1}C_N\, \right)^{-1}A^{\top}+Q
\eeq
and hence, for each fixed $N$, we have $P_N \geq Q$.
Since in the original GDFA model $ \wb$ is idiosyncratic, the noise covariances $R_N$ are uniformly bounded, i.e.  there exists $\alpha$ (independent of $N$) such that $R_N\leq \alpha I$.
Therefore, 
$$
P_N^{-1}+C_N^{\top}R_N^{-1}C_N\geq \frac{1}{\alpha} C_N^{\top} C_N \geq \frac{\lambda_{min}[C_N^{\top} C_N]}{\alpha}I
$$ 
and, as a consequence, 
\beq\label{pintendea0}
\left( \,P_N^{-1}+C_N^{\top}R_N^{-1}C_N\, \right)^{-1}\leq \frac{\alpha I}{\lambda_{min}[C_N^{\top} C_N]}\stackrel{N\rightarrow\infty}{\longrightarrow}0.
\eeq
This inequality  together with \eqref{ARE1} implies that $P_N$ converges monotonically to $Q$.
Hence, the  perturbation  term  $C_NP_NC_N^{\top}$ of $R_N$ in \eqref{lambdan}   must have at least one eigenvalue tending to infinity (actually as many as the rank of $Q$) and must therefore  tend to an unbounded operator so that  $\eb_N(t)$ is \textit{not} an idiosyncratic process. In conclusion, the innovation  model \eqref{Predittore} does \textit{not} satisfy the conditions of a GDFA model.
\qed

Since, as shown in the previous proof, $P_N$ converges to $Q$, which is not the zero matrix, we have the following corollary.
\begin{corollary}\label{CorPE}
Under the  assumptions of Theorem \ref{th:kalman_pred}, the steady state prediction error of the state does not converge to zero (in mean square) as $N$ diverges.
In particular, the one step ahead predictor of the {\em common component vector} $\chib_N(t):= C_N\xb(t)$ does not converge neither to $C\xb(t)$ nor to the measured signal $\yb(t)$, as $N\to \infty$. 

\end{corollary}
 \proof
The one step ahead predictor of $\chib_N(t)$,  $\hat \chib(t\mid t-1) :=C_N \hat\xb_N(t\mid t-1)$ is just the  one step ahead predictor of $\yb_N(t)$ and one can write
$$
\yb_N(t)=\hat \chib_N(t\mid t-1) + \eb_N(t)
$$
where  $\eb_N$ is the output prediction error i.e. the innovation. Since, as we have shown,  the 
 covariance matrix   of the prediction error $\eb_N$ becomes  {\em unbounded} as $N\to \infty$,  the predictor $\hat \chib_N(t\mid t-1)$ cannot be  consistent in mean square. In fact, as we have already seen, the term   $C_NP_NC_N^{\top}$, namely the steady state  covariance matrix
of the prediction error $C_N\xb(t)-C_N \hat\xb_N(t\mid t-1)$  must have at least one eigenvalue tending to infinity and there must then be at least one direction along which the error covariance diverges in mean square. 
This is {\em a fortiori} true for the covariance matrix  $C_NP_NC_N^{\top} +R_N$ of the difference
$\yb_N(t)-\hat \chib_N(t\mid t-1)$.\hfill 
$\Box$

Notice that while $\eb_N(t)$ is \textit{not} idiosyncratic, the ``model noise" process $\tilde{\vb}_N(t):=K_N \eb_N(t)$
does converge in mean-square to a finite covariance noise $\tilde{\vb}(t)$.
In fact its covariance is $\tilde{Q}_N:= K_N  \Lambda_N K_N\tp$ and by rearranging \eqref{ARE} one immediately sees that
$$
\tilde{Q}_N=AP_NA\tp -P_N +Q
$$
which clearly converges to $AQA\tp$ because $P_N\to Q$. Therefore, the (steady-state) covariance of
$\hat\xb_N(t|t-1)$ converges to the unique solution of the Stein equation $$\Sigma =A\Sigma A\tp+ AQA\tp.$$

\begin{remark}
{\em That  the innovation model \eqref{Predittore} cannot in the limit be interpreted as  a valid GDFA model has   important consequences. Among them it 
suggests that 
the application of standard subspace identification methods to the identification of GDFA models 
(see e.g.  \cite{Marcellino-17}), while effective in deriving a good generative model for the observed data,
may lack the capability of extracting the very GDFA feature of the model itself.
In fact, 
these methods are based (typically, via   standard canonical correlation analysis of the future onto the strict past of the process $\yb(t)$)  on the  construction of  a particular basis on the predictor space,   \cite{LPBook}.
Indeed, the results of the following section seem to suggest that considering the pure filter estimator in place of the predictor, may be advantageous.
} 
\end{remark}

\begin{remark}
{\em
We may attempt a frequency-domain analysis of our result.
Obviously  the matrix transfer function of each model of the type  \eqref{Predittore}   is the unique\footnote{Up to uninteresting multiplication on the right
side by an orthogonal matrix.} outer (i.e. stable and minimum-phase)  spectral factor of the spectral density of the truncated process $\yb_N(t)$. For $N\to \infty$   these spectral factors (which are $N\times N$ rational outer matrices of full rank $N$) have  zeros inside the open unit circle. Hence, the  limit spectral factor will have all of its zeros  (at most) in the closed unit circle. Therefore the limit spectral factor could also be  called outer, or (weakly) minimum phase. However, since the innovation process corresponding to this factor is not idiosyncratic, the limit model cannot be a DGFA model. This could be stated by saying that {\em there cannot exist prediction-error innovation models in the  class of GDFA descriptions}.
} 
\end{remark}

\section{The Dynamic case: the pure filter estimator}
 
Since the output noise covariance of the one step-ahead  innovation model is in a sense ``too big" as it has an unbounded component, one may wonder if to obtain a viable GDFA model, one could choose  the (pure) filter  estimate $\E [ \xb(t) \mid  \yb^t]$  instead of  the one-step ahead state predictor. It is in fact well-known that this estimate leads to a model with smaller state error variance.

To verify this conjecture,  consider the steady state estimate $\hat \xb_N(t) :=\E [ \xb(t) \mid  \yb_N^t]  $ given the infinite past, which satisfies the recursion
$$
 \hat \xb_N(t+1)= \E [ \xb(t+1) \mid  \yb_N^{t}] + \E [ \xb(t+1) \mid  \eb_N(t+1)\,] 
$$
which for a model with uncorrelated state and output noises, yields
$$
\hat \xb_N(t+1)= A \hat \xb_N(t)+ L_N \eb_N(t+1)\, 
$$ 
where $L_N:=P_N C_N^{\top}\Lambda_N^{-1}$. This yields the  {\em filtered innovation model}
\beq \label{Innov2}
\left\{
\begin{array}{ll}
 \hat {\xb}_N(t+1)&=A \hat{\xb}_N(t) +L_N \eb_N(t+1) \\
\yb_N(t) & = C_N\hat{\xb}_N(t) + \hat\eb_N(t)   
\end{array}
\right.
\eeq
where (note the hatted symbol) $\hat \eb_N:= \yb_N(t) -C_N\hat{\xb}_N(t)$ is the {\em filter innovation} which is a white noise process. In fact $\eb_N$ and $\hat \eb_N$ are related by the formula
\beq \label{InnInnov2}
\eb_N(t)= [\, I- C_NL_N\,]^{-1} \hat \eb_N(t)
\eeq
which follows from    \eqref{Innov2} as
\bea\nonumber
\hat \eb_N(t)&:= &\yb_N(t) -C_N (\hat \xb_N(t\mid t-1) +L_N \eb_N(t))\\
\nonumber
&=& \eb_N(t) - C_NL_N\eb_N(t)\,.
\eea
This agrees with the fact  that the noise term in the second equation of \eqref{Innov2} is uncorrelated with $ \yb_N^t$ and hence with $\hat \xb_N(t)$.   The variance of $\hat\eb_N(t) $ has the representation
\begin{align*}
\hat \Lambda_N &:= \E[ \hat\eb_N(t)\hat\eb_N(t)^\top]\\
&=  [\, I- C_NP_N C_N^{\top}\Lambda_N^{-1}\,] \Lambda_N  [\, I- C_NP_N C_N^{\top}\Lambda_N^{-1}\,]^{\top}  \\
&= R_N\Lambda_N^{-1}R_N\,.
\end{align*}
\begin{theorem}\label{th:kalman_filt}
Consider  a class of truncated models of the form \eqref{Modclass} with  Assumptions (\ref{primaass}) to (\ref{Ass5}).
Then, the state and output noises in the associated model \eqref{Innov2} are  uncorrelated and the output noise variance $\hat \Lambda_N$ tends to a bounded operator as $N\to \infty$. Therefore \eqref{Innov2} converges to a legitimate DGFA representation of the process $\yb$.
\end{theorem}
\proof That $\vb_N(t) := \eb_N(t+1)$ and $\wb_N(t):= \hat \eb_N(t)$ are uncorrelated follows readily from the equation \eqref{InnInnov2} because $\hat \eb_N(t)$ is white. Moreover, all the eigenvalues of $\Lambda_N=C_NP_NC_N^{\top}+ R_N$ are positive and bounded below. Hence, $\hat \Lambda_N$ 
remains bounded for $N\to \infty$.
\qed

A remarkable property of the pure filter realization which follows already from  the calculation in \eqref{pintendea0} of the previous paragraph, is recast in the following statement.
\begin{corollary}
Under the  assumptions of Theorem \ref{th:kalman_filt},
consider  the (steady-state) covariance matrix of the state filtering error
\begin{align*}
\Pi_N &:= \E\,\left[ \xb(t)-\hat\xb_N(t)\right]\left[ \xb(t)-\hat\xb_N(t)\right]^{\top} \\
&=\left[ \,P_N^{-1} +C_N^{\top} R_N^{-1}C_N\,\right]^{-1}.
\end{align*}
Then, as $N\to \infty$, the filter error  covariance $\Pi_N$ converges to the zero matrix . 
\end{corollary}
This result implies that the limit for $N\to \infty$ of the (steady-state) filtered state estimate must  converge to the true state $\xb(t)$, in other words we may say that the filtered estimate is a {\em consistent estimator}. This should not be surprising since  it agrees with the previous general observation that in any bona-fide GDFA model \eqref{Modclass}     the hidden variable $\xb(t)$ can asymptotically be recovered exactly as a linear functional of the infinite cross sectional  history of the process.

\begin{remark} \label{remark4}
{\em Observe that even if  
the filtered state estimation error converges to zero (in mean square), in general we cannot recover the original idiosyncratic noise.
In fact, if the output noise covariance is uniformly coercive (that is there is a $c>0$ independent of $N$ such that
$R_N \geq cI_N$) then the steady-state covariance of  $\hat{{\boldsymbol{\delta}}}_N(t):=\hat{\eb}_N(t)-\wb_N(t)$   does not converge to zero as it is a 
rank $n$ matrix whose $n$ non-zero eigenvalues are bounded from below by a positive constant.
However, by using also in this case the arguments developed for the static case,  we can show that if $C_N$ is uniformly bounded  
then for each fixed $i$ the $i-th$ component of $\hat{{\boldsymbol{\delta}}}_N(t)$ converges to zero in mean square.}
\end{remark} 

\section{Examples}
We consider the problem of estimating the average concentration over time of two pollutants, namely benzene (C6H6) and carbon monoxide (CO),  in a certain city by means of a large number of sensors spread all over the area. This situation can be described by a DGFA model where the hidden factor vector describes these concentrations while the observed variables models the measurements taken from the sensors.\\
More in details, let $ \xb_1 (t)$ and $ \xb_2 (t)$ describe the average concentration in the city of C6H6 (in $10 \mu g/m^3$) and of CO (in $mg/m^3$), respectively, with sampling time equal to 1 hour, and let $\xb (t) = [\xb_1 (t), \;  \xb_2 (t) ] \tp $ be the vector of latent factors.
A good representation of the dynamics of the factors is given by the first equation of \eqref{Modclass}
where the state matrix $A$ and the the covariance  $Q$ of the white Gaussian noise  $\{ \vb (t)\} $ are
$$A = \bmat 0.9692  & -0.0442  \\ 0.2582  &  0.7707 \emat, \qquad Q =  \bmat 0.1682 &	0.2806 \\ 0.2806 &	0.7531 \emat. $$
This model has been identified  from the time series collected in the period 11 March 2004–3 April 2005 by the regional environmental protection agency (ARPA) within an Italian city  (for more details see \cite{DEVITO2008}). 
Suppose that $N$ sensors are available and that each sensor measures either the concentration of C6H6 or CO; the placement of the sensors is shown in Figure \ref{fig:sensor_net}. For simplicity, we assume that $N$ is a multiple of 4. 
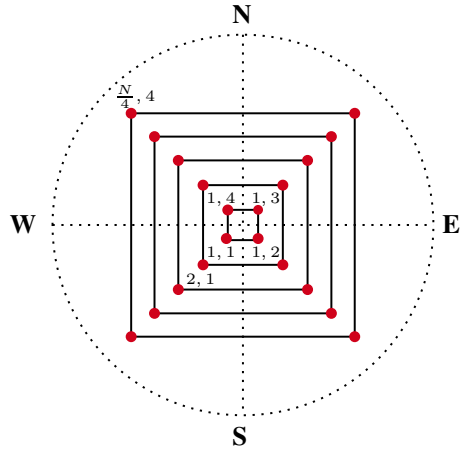
\begin{figure}[h!] \label{fig:sensor_net}
	\centering
	\tikzset{every picture/.style={line width=0.75pt}} 
	
	\begin{tikzpicture}[x=0.75pt,y=0.75pt,yscale=-0.73,xscale=0.73]
		
		\draw   (274.54,132) -- (363.84,132) -- (363.84,221.3) -- (274.54,221.3) -- cycle ;
		\draw   (258.09,115.55) -- (380.29,115.55) -- (380.29,237.75) -- (258.09,237.75) -- cycle ;
		\draw   (291.64,149.1) -- (346.75,149.1) -- (346.75,204.2) -- (291.64,204.2) -- cycle ;
		\draw   (308.72,166.18) -- (329.67,166.18) -- (329.67,187.12) -- (308.72,187.12) -- cycle ;
		\draw  [color={rgb, 255:red, 208; green, 2; blue, 27 }  ,draw opacity=1 ][fill={rgb, 255:red, 208; green, 2; blue, 27 }  ,fill opacity=1 ] (377.19,237.75) .. controls (377.19,236.04) and (378.58,234.65) .. (380.29,234.65) .. controls (382.01,234.65) and (383.39,236.04) .. (383.39,237.75) .. controls (383.39,239.46) and (382.01,240.85) .. (380.29,240.85) .. controls (378.58,240.85) and (377.19,239.46) .. (377.19,237.75) -- cycle ;
		\draw  [color={rgb, 255:red, 208; green, 2; blue, 27 }  ,draw opacity=1 ][fill={rgb, 255:red, 208; green, 2; blue, 27 }  ,fill opacity=1 ] (326.97,166.18) .. controls (326.97,167.67) and (328.18,168.88) .. (329.67,168.88) .. controls (331.16,168.88) and (332.37,167.67) .. (332.37,166.18) .. controls (332.37,164.69) and (331.16,163.48) .. (329.67,163.48) .. controls (328.18,163.48) and (326.97,164.69) .. (326.97,166.18) -- cycle ;
		\draw  [color={rgb, 255:red, 208; green, 2; blue, 27 }  ,draw opacity=1 ][fill={rgb, 255:red, 208; green, 2; blue, 27 }  ,fill opacity=1 ] (377.19,115.55) .. controls (377.19,113.84) and (378.58,112.45) .. (380.29,112.45) .. controls (382.01,112.45) and (383.39,113.84) .. (383.39,115.55) .. controls (383.39,117.26) and (382.01,118.65) .. (380.29,118.65) .. controls (378.58,118.65) and (377.19,117.26) .. (377.19,115.55) -- cycle ;
		\draw  [color={rgb, 255:red, 208; green, 2; blue, 27 }  ,draw opacity=1 ][fill={rgb, 255:red, 208; green, 2; blue, 27 }  ,fill opacity=1 ] (254.99,115.55) .. controls (254.99,113.84) and (256.38,112.45) .. (258.09,112.45) .. controls (259.81,112.45) and (261.19,113.84) .. (261.19,115.55) .. controls (261.19,117.26) and (259.81,118.65) .. (258.09,118.65) .. controls (256.38,118.65) and (254.99,117.26) .. (254.99,115.55) -- cycle ;
		\draw  [color={rgb, 255:red, 208; green, 2; blue, 27 }  ,draw opacity=1 ][fill={rgb, 255:red, 208; green, 2; blue, 27 }  ,fill opacity=1 ] (254.99,237.75) .. controls (254.99,236.04) and (256.38,234.65) .. (258.09,234.65) .. controls (259.81,234.65) and (261.19,236.04) .. (261.19,237.75) .. controls (261.19,239.46) and (259.81,240.85) .. (258.09,240.85) .. controls (256.38,240.85) and (254.99,239.46) .. (254.99,237.75) -- cycle ;
		\draw  [color={rgb, 255:red, 208; green, 2; blue, 27 }  ,draw opacity=1 ][fill={rgb, 255:red, 208; green, 2; blue, 27 }  ,fill opacity=1 ] (271.44,132) .. controls (271.44,130.29) and (272.83,128.9) .. (274.54,128.9) .. controls (276.26,128.9) and (277.64,130.29) .. (277.64,132) .. controls (277.64,133.71) and (276.26,135.1) .. (274.54,135.1) .. controls (272.83,135.1) and (271.44,133.71) .. (271.44,132) -- cycle ;
		\draw  [color={rgb, 255:red, 208; green, 2; blue, 27 }  ,draw opacity=1 ][fill={rgb, 255:red, 208; green, 2; blue, 27 }  ,fill opacity=1 ] (326.57,186.02) .. controls (326.57,184.31) and (327.95,182.92) .. (329.67,182.92) .. controls (331.38,182.92) and (332.77,184.31) .. (332.77,186.02) .. controls (332.77,187.73) and (331.38,189.12) .. (329.67,189.12) .. controls (327.95,189.12) and (326.57,187.73) .. (326.57,186.02) -- cycle ;
		\draw  [color={rgb, 255:red, 208; green, 2; blue, 27 }  ,draw opacity=1 ][fill={rgb, 255:red, 208; green, 2; blue, 27 }  ,fill opacity=1 ] (343.65,204.2) .. controls (343.65,202.49) and (345.03,201.1) .. (346.75,201.1) .. controls (348.46,201.1) and (349.85,202.49) .. (349.85,204.2) .. controls (349.85,205.92) and (348.46,207.3) .. (346.75,207.3) .. controls (345.03,207.3) and (343.65,205.92) .. (343.65,204.2) -- cycle ;
		\draw  [color={rgb, 255:red, 208; green, 2; blue, 27 }  ,draw opacity=1 ][fill={rgb, 255:red, 208; green, 2; blue, 27 }  ,fill opacity=1 ] (360.74,221.3) .. controls (360.74,219.59) and (362.13,218.2) .. (363.84,218.2) .. controls (365.56,218.2) and (366.94,219.59) .. (366.94,221.3) .. controls (366.94,223.01) and (365.56,224.4) .. (363.84,224.4) .. controls (362.13,224.4) and (360.74,223.01) .. (360.74,221.3) -- cycle ;
		\draw  [color={rgb, 255:red, 208; green, 2; blue, 27 }  ,draw opacity=1 ][fill={rgb, 255:red, 208; green, 2; blue, 27 }  ,fill opacity=1 ] (271.44,221.3) .. controls (271.44,219.59) and (272.83,218.2) .. (274.54,218.2) .. controls (276.26,218.2) and (277.64,219.59) .. (277.64,221.3) .. controls (277.64,223.01) and (276.26,224.4) .. (274.54,224.4) .. controls (272.83,224.4) and (271.44,223.01) .. (271.44,221.3) -- cycle ;
		\draw  [color={rgb, 255:red, 208; green, 2; blue, 27 }  ,draw opacity=1 ][fill={rgb, 255:red, 208; green, 2; blue, 27 }  ,fill opacity=1 ] (360.74,132) .. controls (360.74,130.29) and (362.13,128.9) .. (363.84,128.9) .. controls (365.56,128.9) and (366.94,130.29) .. (366.94,132) .. controls (366.94,133.71) and (365.56,135.1) .. (363.84,135.1) .. controls (362.13,135.1) and (360.74,133.71) .. (360.74,132) -- cycle ;
		\draw  [color={rgb, 255:red, 208; green, 2; blue, 27 }  ,draw opacity=1 ][fill={rgb, 255:red, 208; green, 2; blue, 27 }  ,fill opacity=1 ] (304.62,186.02) .. controls (304.62,184.31) and (306.01,182.92) .. (307.72,182.92) .. controls (309.43,182.92) and (310.82,184.31) .. (310.82,186.02) .. controls (310.82,187.73) and (309.43,189.12) .. (307.72,189.12) .. controls (306.01,189.12) and (304.62,187.73) .. (304.62,186.02) -- cycle ;
		\draw  [color={rgb, 255:red, 208; green, 2; blue, 27 }  ,draw opacity=1 ][fill={rgb, 255:red, 208; green, 2; blue, 27 }  ,fill opacity=1 ] (305.62,166.28) .. controls (305.62,164.57) and (307.01,163.18) .. (308.72,163.18) .. controls (310.43,163.18) and (311.82,164.57) .. (311.82,166.28) .. controls (311.82,167.99) and (310.43,169.38) .. (308.72,169.38) .. controls (307.01,169.38) and (305.62,167.99) .. (305.62,166.28) -- cycle ;
		\draw  [color={rgb, 255:red, 208; green, 2; blue, 27 }  ,draw opacity=1 ][fill={rgb, 255:red, 208; green, 2; blue, 27 }  ,fill opacity=1 ] (343.65,149.1) .. controls (343.65,147.39) and (345.03,146) .. (346.75,146) .. controls (348.46,146) and (349.85,147.39) .. (349.85,149.1) .. controls (349.85,150.81) and (348.46,152.2) .. (346.75,152.2) .. controls (345.03,152.2) and (343.65,150.81) .. (343.65,149.1) -- cycle ;
		\draw  [color={rgb, 255:red, 208; green, 2; blue, 27 }  ,draw opacity=1 ][fill={rgb, 255:red, 208; green, 2; blue, 27 }  ,fill opacity=1 ] (288.54,204.2) .. controls (288.54,202.49) and (289.93,201.1) .. (291.64,201.1) .. controls (293.35,201.1) and (294.74,202.49) .. (294.74,204.2) .. controls (294.74,205.92) and (293.35,207.3) .. (291.64,207.3) .. controls (289.93,207.3) and (288.54,205.92) .. (288.54,204.2) -- cycle ;
		\draw  [color={rgb, 255:red, 208; green, 2; blue, 27 }  ,draw opacity=1 ][fill={rgb, 255:red, 208; green, 2; blue, 27 }  ,fill opacity=1 ] (288.54,149.1) .. controls (288.54,147.39) and (289.93,146) .. (291.64,146) .. controls (293.35,146) and (294.74,147.39) .. (294.74,149.1) .. controls (294.74,150.81) and (293.35,152.2) .. (291.64,152.2) .. controls (289.93,152.2) and (288.54,150.81) .. (288.54,149.1) -- cycle ;
		\draw   (241.97,99.43) -- (396.42,99.43) -- (396.42,253.88) -- (241.97,253.88) -- cycle ;
		\draw  [color={rgb, 255:red, 208; green, 2; blue, 27 }  ,draw opacity=1 ][fill={rgb, 255:red, 208; green, 2; blue, 27 }  ,fill opacity=1 ] (238.87,253.88) .. controls (238.87,252.16) and (240.26,250.78) .. (241.97,250.78) .. controls (243.68,250.78) and (245.07,252.16) .. (245.07,253.88) .. controls (245.07,255.59) and (243.68,256.98) .. (241.97,256.98) .. controls (240.26,256.98) and (238.87,255.59) .. (238.87,253.88) -- cycle ;
		\draw  [color={rgb, 255:red, 208; green, 2; blue, 27 }  ,draw opacity=1 ][fill={rgb, 255:red, 208; green, 2; blue, 27 }  ,fill opacity=1 ] (393.32,253.88) .. controls (393.32,252.16) and (394.71,250.78) .. (396.42,250.78) .. controls (398.13,250.78) and (399.52,252.16) .. (399.52,253.88) .. controls (399.52,255.59) and (398.13,256.98) .. (396.42,256.98) .. controls (394.71,256.98) and (393.32,255.59) .. (393.32,253.88) -- cycle ;
		\draw  [color={rgb, 255:red, 208; green, 2; blue, 27 }  ,draw opacity=1 ][fill={rgb, 255:red, 208; green, 2; blue, 27 }  ,fill opacity=1 ] (393.32,99.43) .. controls (393.32,97.71) and (394.71,96.33) .. (396.42,96.33) .. controls (398.13,96.33) and (399.52,97.71) .. (399.52,99.43) .. controls (399.52,101.14) and (398.13,102.53) .. (396.42,102.53) .. controls (394.71,102.53) and (393.32,101.14) .. (393.32,99.43) -- cycle ;
		\draw  [color={rgb, 255:red, 208; green, 2; blue, 27 }  ,draw opacity=1 ][fill={rgb, 255:red, 208; green, 2; blue, 27 }  ,fill opacity=1 ] (238.87,99.43) .. controls (238.87,97.71) and (240.26,96.33) .. (241.97,96.33) .. controls (243.68,96.33) and (245.07,97.71) .. (245.07,99.43) .. controls (245.07,101.14) and (243.68,102.53) .. (241.97,102.53) .. controls (240.26,102.53) and (238.87,101.14) .. (238.87,99.43) -- cycle ;
		\draw  [dash pattern={on 0.84pt off 2.51pt}]  (319.19,45.09) -- (319.19,308.21) ;
		\draw  [dash pattern={on 0.84pt off 2.51pt}]  (187.63,176.65) -- (450.76,176.65) ;
		\draw  [dash pattern={on 0.84pt off 2.51pt}] (187.63,176.65) .. controls (187.63,103.99) and (246.53,45.09) .. (319.19,45.09) .. controls (391.85,45.09) and (450.76,103.99) .. (450.76,176.65) .. controls (450.76,249.31) and (391.85,308.21) .. (319.19,308.21) .. controls (246.53,308.21) and (187.63,249.31) .. (187.63,176.65) -- cycle ;
		
		\draw (292,190) node [anchor=north west][inner sep=0.75pt]  [font=\small] [align=left] {{\tiny $1,1$}};
		\draw (323,152) node [anchor=north west][inner sep=0.75pt]  [font=\small] [align=left] {{\tiny $1,3$}};
		\draw (292,152) node [anchor=north west][inner sep=0.75pt]  [font=\small] [align=left] {{\tiny $1,4$}};
		\draw (278,208) node [anchor=north west][inner sep=0.75pt]  [font=\small] [align=left] {{\tiny $2,1$}};
		\draw (323,190) node [anchor=north west][inner sep=0.75pt]  [font=\small] [align=left] {{\tiny  $1,2$}};
		\draw (228,78) node [anchor=north west][inner sep=0.75pt]  [font=\small] [align=left] {{\tiny $\frac{N}{4},4$}};
		\draw (310,23) node [anchor=north west][inner sep=0.75pt]  [font=\normalsize] [align=left] {{\fontfamily{ptm}\selectfont \textbf{N}}};
		\draw (310,314) node [anchor=north west][inner sep=0.75pt]  [font=\normalsize] [align=left] {{\fontfamily{ptm}\selectfont \textbf{S}}};
		\draw (156,168) node [anchor=north west][inner sep=0.75pt]  [font=\normalsize] [align=left] {{\fontfamily{ptm}\selectfont \textbf{W}}};
		\draw (455,168) node [anchor=north west][inner sep=0.75pt]  [font=\normalsize] [align=left] {{\fontfamily{ptm}\selectfont \textbf{E}}};

	\end{tikzpicture}
	\caption{ Placement of the $N$ sensors all over the city. The sensors, represented by red circles, are located on the vertices of concentric squares. In particular, the $(l,k)$ sensor is situated on the $k$-th vertex of the $l$-th square. }
\end{figure}
Since the concentration of the pollutants varies considerably within the city, the sensor output at time $t$ is a measure of the average concentration corrupted by a local random fluctuation related to its position. In addition, each sensor is affected by an accidental measurement error which is independent from the other sensors.
Then, we can describe the observation process by the second equation of \eqref{Modclass} 
where  $C_N  = \bmat 1 & 0 & 1 &  0 & 1 & 0 & \dots  \\ 0 & 1 &  0 &  1 & 0 &  1 &\dots \emat \tp $
and the idiosyncratic process $\{\wb_N (t)\} $ is modeled as follows.
We define the idiosyncratic noise vector as $\wb_N (t) = \bmat \wb_{1,1}(t), & \dots, &  \wb_{1,4}(t) ,&   \wb_{2,1 }(t), \dots , \wb_{\frac{N}{4},4 }(t) \emat \tp  $
and we assume that the component $\wb_{l,k}(t)$ is given by a term proportional to the average of the noise affecting the  sensors located in the preceding square plus an uncorrelated white Gaussian noise. 
Mathematically, we have for  $k = 1, \dots, 4$ and $l = 2,3, \dots, N/4$ 
\begin{align*}
w_{1,k} (t) &= n_{1,k}(t) ,\\
w_{l,k} (t) &= \frac{0.5}{4}\sum_{h=0}^{4} w_{l-1,h} + n_{l,k}(t).
\end{align*}
Here, $\{\nb_{l,k}(t)\}$ is a normalized white Gaussian noise uncorrelated at all times with $\{ \vb (t)\} $  and with $\{ \nb_{\bar k, \bar l} (t)\} $ for $(k,l) \neq (\bar k, \bar l) .$ \\
We estimate the latent variable $\xb (t) $  by the Kalman one-step ahead predictor \eqref{Predittore} and the pure filter \eqref{Innov2} for an increasing number $N$ of sensors. The results are summarized in Figure \ref{fig:errors} and \ref{fig:covariance} and they confirm that the pure filter yields to a consistent estimate, while the state prediction error does not converge to zero as $N$ diverges. 
\begin{figure}
	\centering
\begin{tabular}{l l l}
		\includegraphics[width=0.6\linewidth]{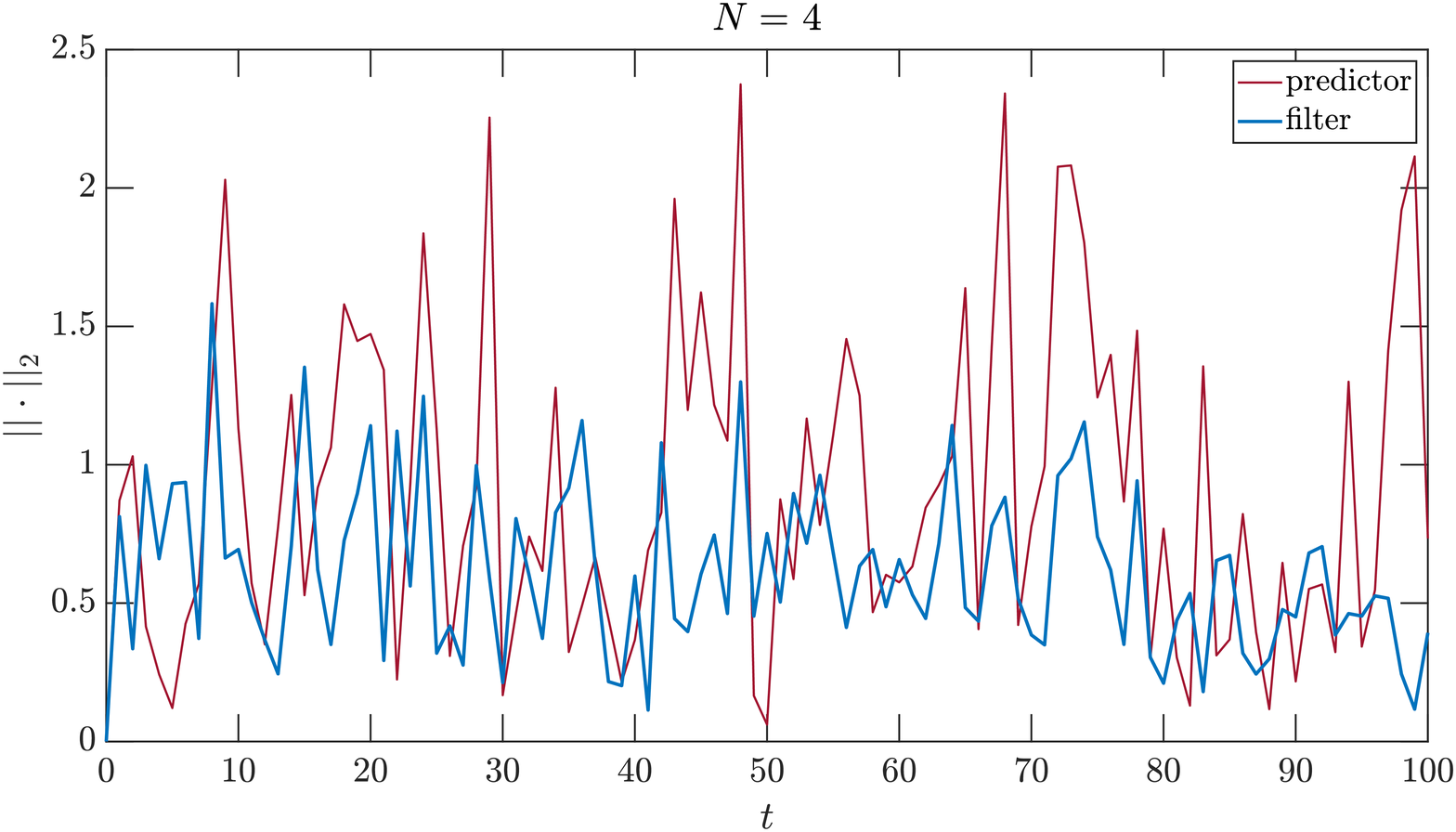}  \\
		\includegraphics[width=0.6\linewidth]{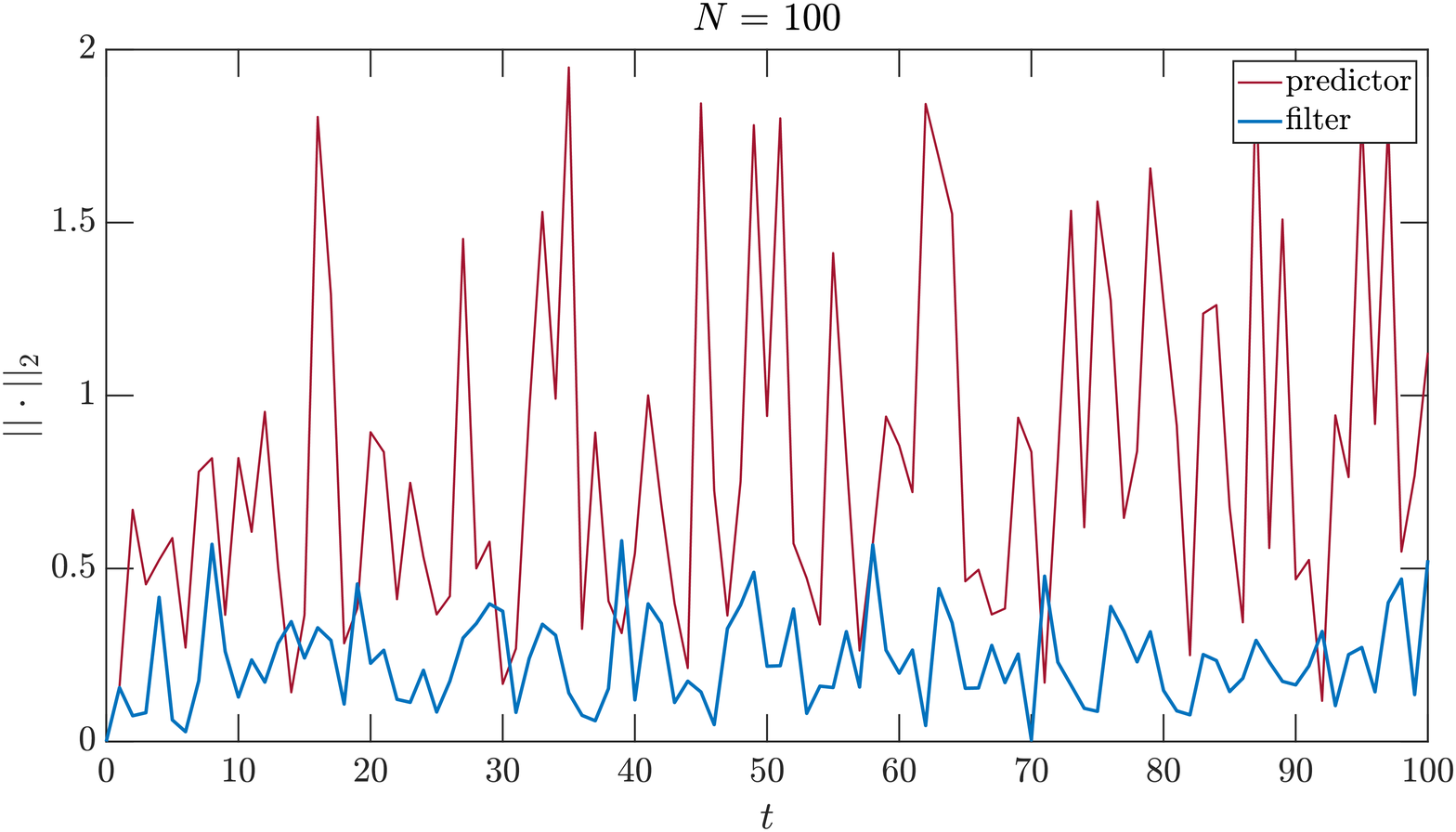} \\
		\includegraphics[width=0.6\linewidth]{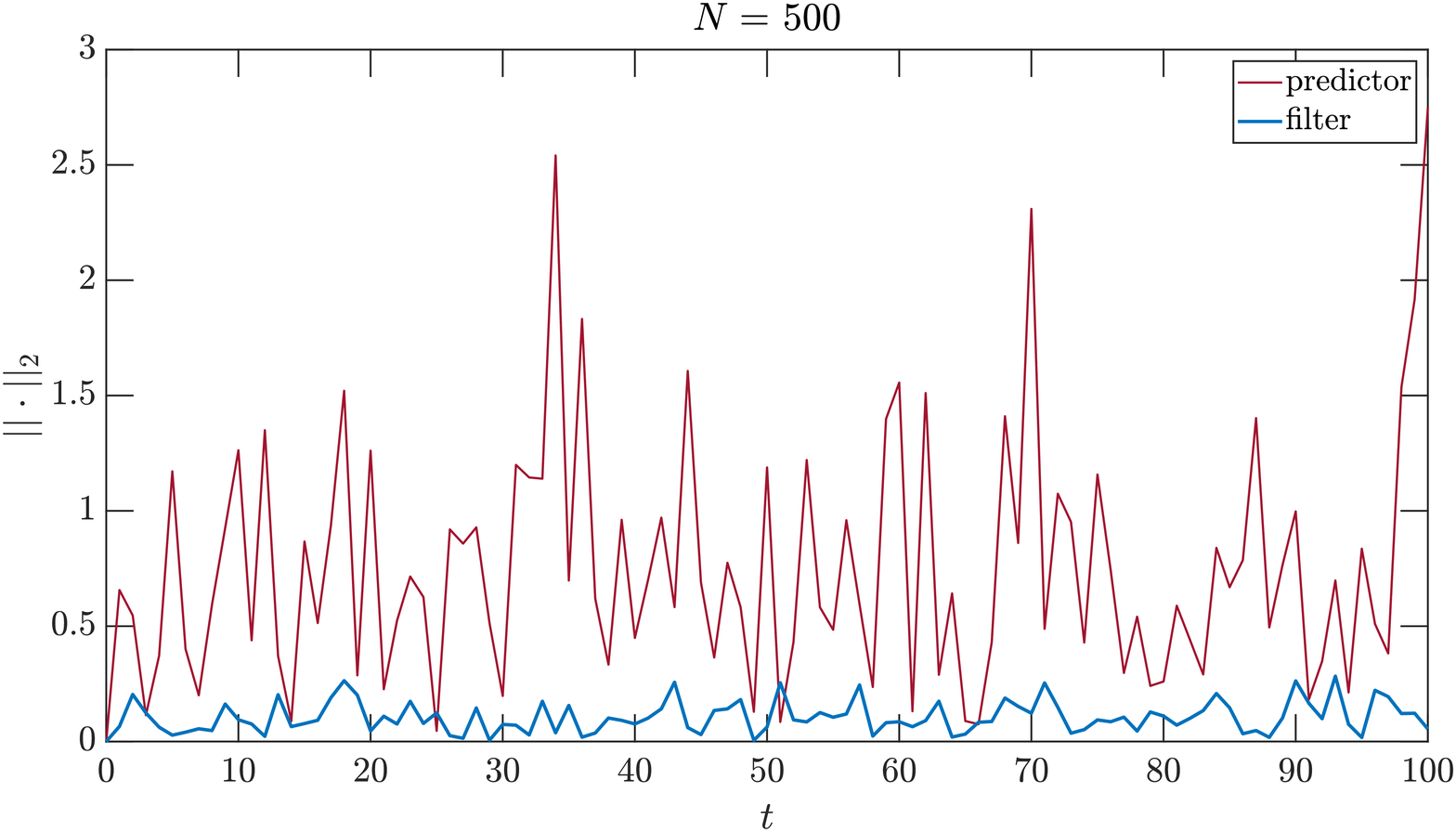} \\
		\includegraphics[width=0.6\linewidth ]{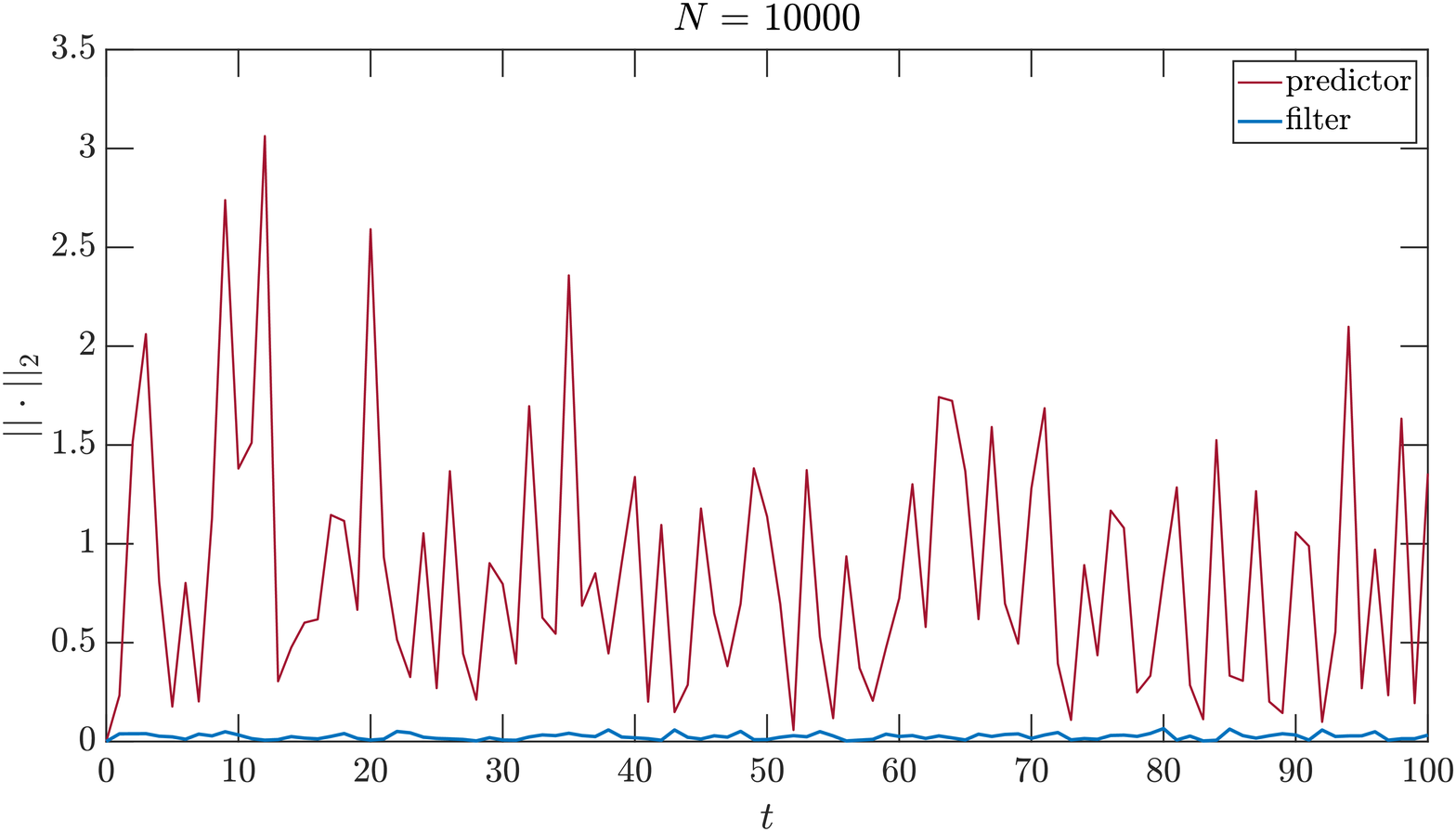} 		
\end{tabular}
\caption{ Euclidean norm of the steady-state prediction error $x(t)-\hat x_N(t|t-1) $ and filtering error $ x(t)-\hat x_N(t) $ for increasing dimension $N$. }
\label{fig:errors}
\end{figure}
\begin{figure}
	\centering
	\includegraphics[width=0.6\linewidth]{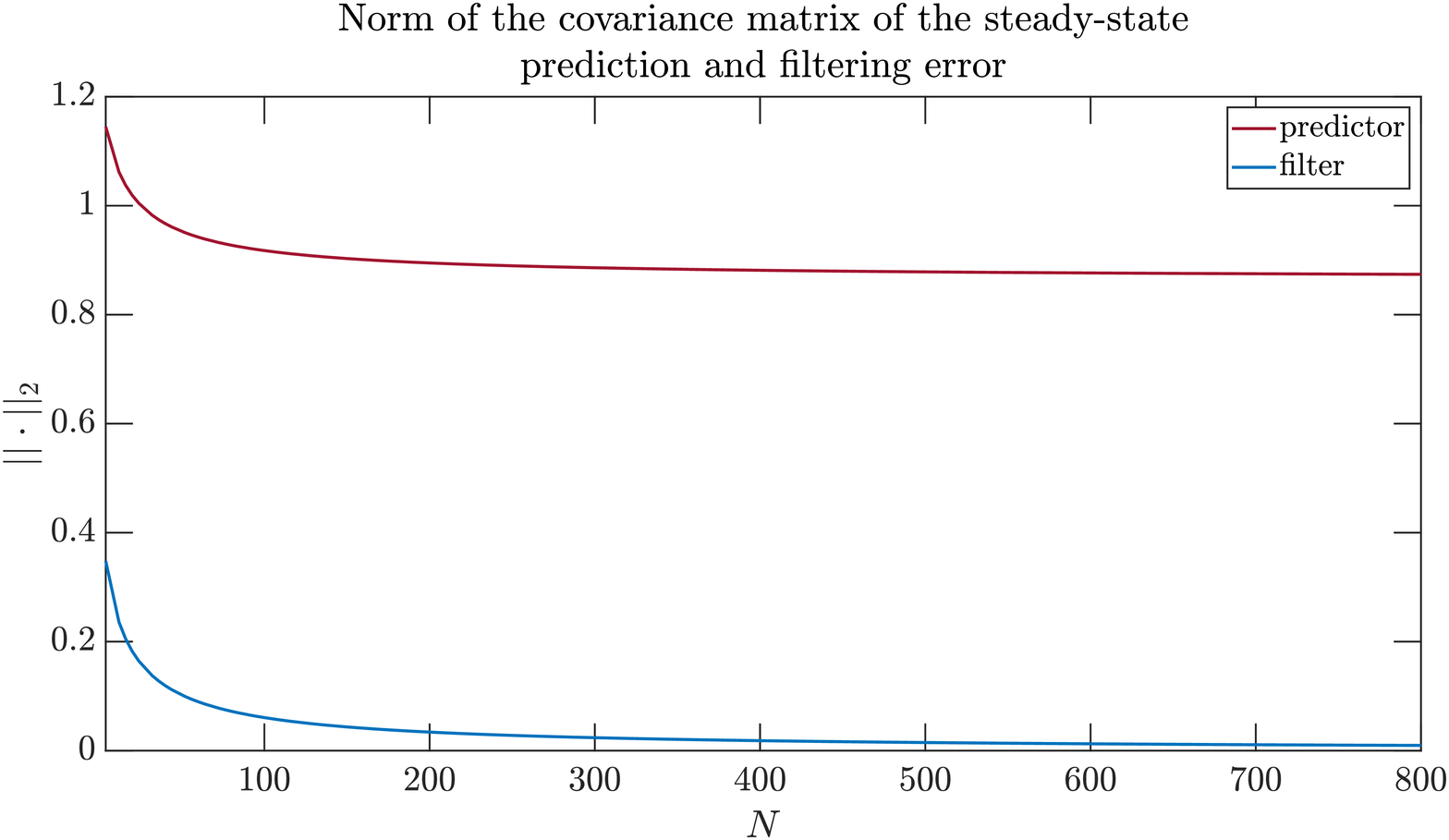}
\caption{ Euclidean norm of the steady-state prediction error covariance matrix $P_N$ and of the steady-state filter error covariance matrix $\Pi_N$ for increasing dimension $N$. }
\label{fig:covariance}
\end{figure}
To better illustrate Remark \ref{remark4}, we also compute the ``Euclidean'' and the ``infinity'' norms of $\hat{{\boldsymbol{\delta}}}_N(t):=\hat{\eb}_N(t)-\wb_N(t)$, i.e.
\begin{align*}
\|\hat{{\boldsymbol{\delta}}}_N(t)\|&:=\sqrt{\E[\hat{{\boldsymbol{\delta}}}_N(t)\tp \hat{{\boldsymbol{\delta}}}_N(t)]}=\sqrt{{\rm tr}[C_N\Pi_NC_N\tp ]} \\
\|\hat{{\boldsymbol{\delta}}}_N(t)\|_\infty &:=\max_i\sqrt{\E[\hat{{\boldsymbol{\delta}}}_N(t)_i]^2}=\max_i \sqrt{[C_N\Pi_NC_N\tp]_{ii}}
\end{align*}
with $\hat{{\boldsymbol{\delta}}}_N(t)_i$ being the $i$-th component of $\hat{{\boldsymbol{\delta}}}_N(t)$
and  $[C_N\Pi_NC_N\tp]_{ii}$ the $i-$th diagonal element of $C_N\Pi_NC_N\tp$. 
The results are shown in Figure \ref{fig:delta_convergence} and they reveal that, according to what observed in Remark \ref{remark4}, the overall vector $\hat{{\boldsymbol{\delta}}}_N(t)$ does not converge to zero in mean square even if  its $i-th$ component does converge to zero for each $i$. 
\begin{figure}
	\centering
	\includegraphics[width=0.6\linewidth]{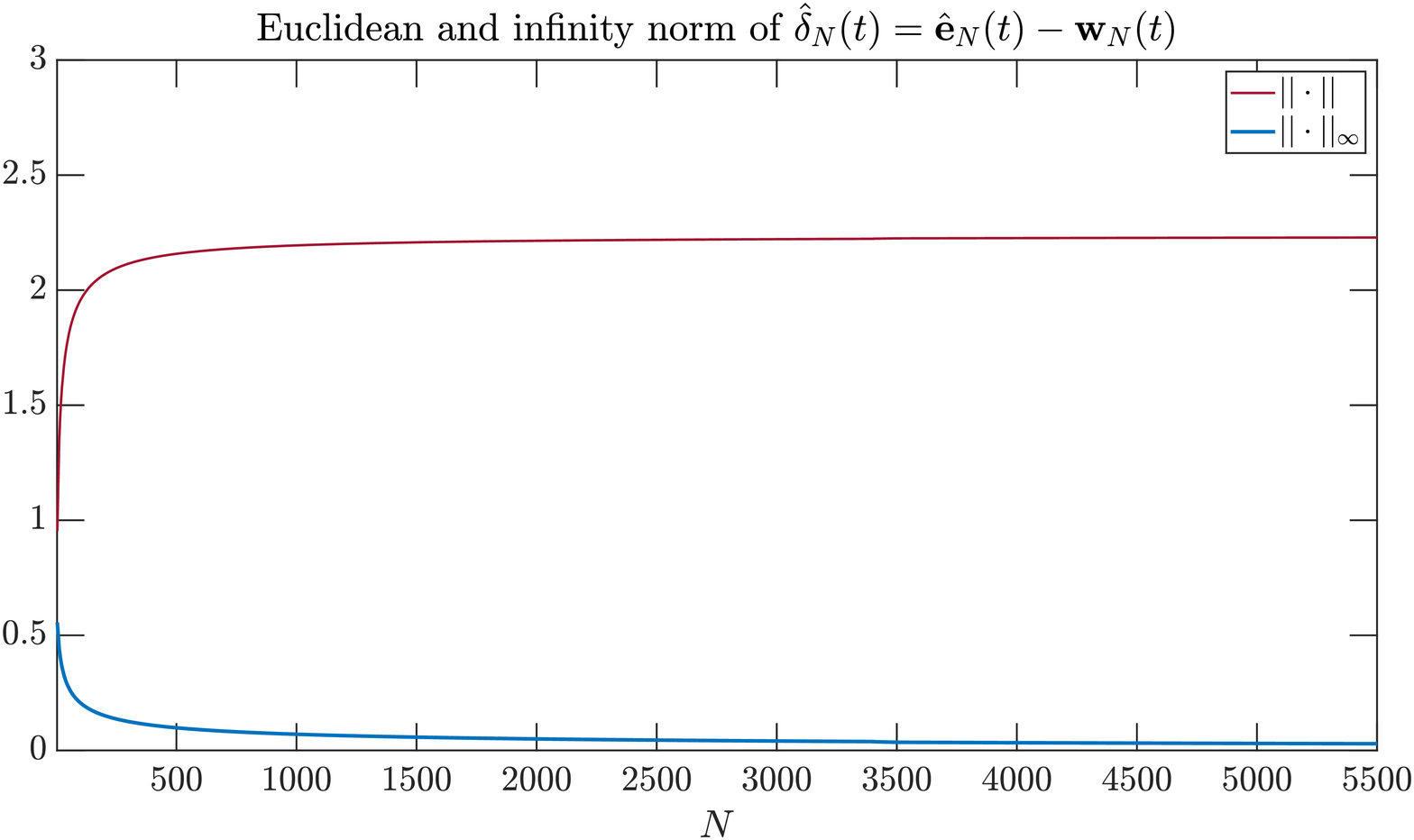}
	\caption{ Euclidean and infinity norm of the random variable $\hat{{\boldsymbol{\delta}}}_N(t) = \hat{\eb}_N(t)-\wb_N(t) $ for increasing dimension $N$. }
	\label{fig:delta_convergence}
\end{figure}

\section{Conclusion}
We have shown that the (pure) Kalman filter leads to a legitimate GDFA model  while the standard Kalman predictor does not. As a consequence only the former  model allows   a consistent estimation of the factor process  as  the cross-sectional dimension tends to infinity.
\appendix

{\bf Appendix}
\section{About the space $\ell^2(\Sigma)$}
We analyse the relation between  $\ell^2$ and $\ell^2(\Sigma)$. In general this relation depends on $\Sigma$. Let us first consider the  case where the components of  $\yb$ are zero-mean uncorrelated  random variables with unbounded variance, say $\Sigma \geq I$ (where $I$ is the identity operator); for example let $\yb$ have uncorrelated components with  $\mathbb{E}  \yb(k)^2 = k^2$.
Then   $\| a \|^2_\Sigma \geq \| a \|^2$ and hence $\ell^2 \supseteq \ell^2(\Sigma)$.
On the other hand, let $a:= \{a(k)=1/k,\,k \in \mathbb{N}\}$;
clearly, $a\in\ell^2$ but $a\not\in\ell^2(\Sigma)$. Hence, in this case, $\ell^2 \supsetneqq \ell^2(\Sigma)$.

When the $\yb(k)$'s are zero-mean random variables with a bounded variance matrix, say $\Sigma \leq I$ we have $\| a \|^2_\Sigma \leq \| a \|^2$ and hence
 $\ell^2 \subseteq \ell^2(\Sigma)$.
Consider in particular the case when $\mathbb{E}  \yb(k)^2 = \frac{1}{k^2}.$
Then, let $a:= \{a(k)=1,\,\forall k \in \mathbb{N}\}$:
clearly, $a\not\in\ell^2$ but $a\in\ell^2(\Sigma)$.
Hence, in this  case, $\ell^2 \subsetneqq \ell^2(\Sigma)$.

\section{The notion of strong linear independence}

In this section we collect some relevant facts discussed in \cite{Bottegal-P-15}
where more details and the proofs of the collected results are available.
Let $C\in\Rbb^{\infty\times n}$ and $c^i, \,i=1,\ldots,n$
be the columns of $C$.
Let $C_N$ be the (finite)
submatrix of $C$ obtained by extracting the first $N$
rows of $C$.
Let $c_N^i$ be the $i$-th column of $C_N$ and define
$$
 \tilde{c}_N^i := c_N^i - \Pi[\, c_N^i \mid \check{\mathcal{C}}_N^i]
$$
where $\Pi$ is the orthogonal projection onto the  Euclidean space $ \check{\mathcal{C}}_N^i= \Span \{c_N^j ,\,j\neq i \, \} $ of dimension (at most) $n-1$.

\begin{definition}
{\em 
The column vectors $c^i, \,i=1,\ldots,n$ in $\Rbb^{\infty}$ are {\bf strongly linearly independent} if
\begin{equation}\label{eq:condition}
\lim_{N\rightarrow\infty}\|\tilde{c}_N^i \|_2 = +\infty\,\qq \,i=1,\ldots,n\,.
\end{equation}}
\end{definition}
If $n=1$ the condition \eqref{eq:condition} is equivalent to $\|c\|_2= \infty$.
In a sense, this condition says that the tails of  two strongly linearly independent vectors in $\Rbb^{\infty}$ cannot get ``too close'' asymptotically.
 \begin{theorem}\label{Thm:Strong}
   Let $\yb$ be a purely deterministic sequence of rank $n$, i.e. let
$$
\yb(k)\,   = \sum_{i=1}^{n} \, c^i(k)\, \xb_i\,,\qquad  k\in \Zbb_+\,;
$$
with uncorrelated latent variables $\xb_i$. Then $\yb$  is   not idiosyncratic  if and only if, the vectors $c^i, \,i=1,\ldots,n$ are strongly linearly independent.
\end{theorem}

\begin{corollary}
$C\in \Rbb^{\infty \times n}$ has strongly linearly independent columns if and only if $ \lim_{N\to \infty} \lambda_{min} [ C_N\tp C_N]= \infty$.
\end{corollary}

\begin{corollary}
The covariance $\Sigma$ has a GFA decomposition with $n$ latent
factors if and only if it can be decomposed as the sum of a  matrix
 $\tilde{\Sigma}$ which defines a bounded operator in $\ell^2$ and a
$\rank\; n$ perturbation $\hat{\Sigma}=C C^{\top}$ where $C\in \Rbb
^{\infty \times n}$ has strongly linearly independent columns.
\end{corollary}

\bibliographystyle{model5-names}      
\bibliography{generalized_factor_models}

\begin{thebibliography}{39}
\expandafter\ifx\csname natexlab\endcsname\relax\def\natexlab#1{#1}\fi
\providecommand{\bibinfo}[2]{#2}
\ifx\xfnm\relax \def\xfnm[#1]{\unskip,\space#1}\fi
\bibitem[{Akhiezer \& Glazman(1961)}]{Akhiezer-G-61}
\bibinfo{author}{Akhiezer, N.}, \& \bibinfo{author}{Glazman, I.~M.}
  (\bibinfo{year}{1961}).
\newblock {\it \bibinfo{title}{Theory of Linear Operators in Hilbert Space Vol
  I}\/}.
\newblock \bibinfo{address}{New York}: \bibinfo{publisher}{Fredrik Ungar Pub.
  Co.}
\bibitem[{Anderson et~al.(2018)Anderson, Braumann \&
  Deistler}]{Anderson-B-D-18}
\bibinfo{author}{Anderson, B. D.~O.}, \bibinfo{author}{Braumann, A.}, \&
  \bibinfo{author}{Deistler, M.} (\bibinfo{year}{2018}).
\newblock \bibinfo{title}{Identification of generalized dynamic factor models
  from mixed-frequency data}.
\newblock In {\it \bibinfo{booktitle}{IFAC Papers On Line 51-15}\/} (pp.
  \bibinfo{pages}{1008--1013}).
\bibitem[{Anderson \& Deistler(2008)}]{Anderson-Deistler_2008}
\bibinfo{author}{Anderson, B. D.~O.}, \& \bibinfo{author}{Deistler, M.}
  (\bibinfo{year}{2008}).
\newblock \bibinfo{title}{Generalized linear dynamic factor models - a
  structure theory}.
\newblock In {\it \bibinfo{booktitle}{Proc. of the IEEE Decision and Control
  conference (CDC)}\/}.
\bibitem[{Bekker \& de~Leeuw(1987)}]{Bekker-L-87}
\bibinfo{author}{Bekker, P.~A.}, \& \bibinfo{author}{de~Leeuw, J.}
  (\bibinfo{year}{1987}).
\newblock \bibinfo{title}{The rank of reduced dispersion matrices}.
\newblock {\it \bibinfo{journal}{Psychometrika}\/},  {\it
  \bibinfo{volume}{52}\/}, \bibinfo{pages}{125--135}.
\bibitem[{{Bottegal} \& {Picci}(2015)}]{Bottegal-P-15}
\bibinfo{author}{{Bottegal}, G.}, \& \bibinfo{author}{{Picci}, G.}
  (\bibinfo{year}{2015}).
\newblock \bibinfo{title}{Modeling complex systems by generalized factor
  analysis}.
\newblock {\it \bibinfo{journal}{IEEE Transactions on Automatic Control}\/},
  {\it \bibinfo{volume}{60}\/}, \bibinfo{pages}{759--774}.
\bibitem[{Burt(1909)}]{burt_1909}
\bibinfo{author}{Burt, C.} (\bibinfo{year}{1909}).
\newblock \bibinfo{title}{Experimental tests of general intelligence}.
\newblock {\it \bibinfo{journal}{British Journal of Psychology}\/},  {\it
  \bibinfo{volume}{3}\/}, \bibinfo{pages}{94--177}.
\bibitem[{Chamberlain(1983)}]{chamberlain_1983}
\bibinfo{author}{Chamberlain, G.} (\bibinfo{year}{1983}).
\newblock \bibinfo{title}{Funds, factors and diversification in arbitrage
  pricing models}.
\newblock {\it \bibinfo{journal}{Econometrica}\/},  {\it
  \bibinfo{volume}{51}\/}, \bibinfo{pages}{1305--1324}.
\bibitem[{Chamberlain \& Rothschild(1983)}]{Chamberlain-R-83}
\bibinfo{author}{Chamberlain, G.}, \& \bibinfo{author}{Rothschild, M.}
  (\bibinfo{year}{1983}).
\newblock \bibinfo{title}{Arbitrage, factor structure and meanvariance analysis
  on large asset markets}.
\newblock {\it \bibinfo{journal}{Econometrica}\/},  {\it
  \bibinfo{volume}{51}\/}, \bibinfo{pages}{1281--1304}.
\bibitem[{Chandrasekaran et~al.(2012)Chandrasekaran, Parrilo \&
  Willsky}]{Chandra-etal-011}
\bibinfo{author}{Chandrasekaran, V.}, \bibinfo{author}{Parrilo, P.~A.}, \&
  \bibinfo{author}{Willsky, A.~S.} (\bibinfo{year}{2012}).
\newblock \bibinfo{title}{Latent variable graphical model selection via convex
  optimization}.
\newblock {\it \bibinfo{journal}{The Annals of Statistics}\/},  {\it
  \bibinfo{volume}{40}\/}, \bibinfo{pages}{1935--1967}.
\bibitem[{{Ciccone} et~al.(2019){Ciccone}, {Ferrante} \&
  {Zorzi}}]{ciccone2017factor}
\bibinfo{author}{{Ciccone}, V.}, \bibinfo{author}{{Ferrante}, A.}, \&
  \bibinfo{author}{{Zorzi}, M.} (\bibinfo{year}{2019}).
\newblock \bibinfo{title}{Factor models with real data: A robust estimation of
  the number of factors}.
\newblock {\it \bibinfo{journal}{IEEE Transactions on Automatic Control}\/},
  {\it \bibinfo{volume}{64}\/}, \bibinfo{pages}{2412--2425}.
\bibitem[{{Ciccone} et~al.(2020){Ciccone}, {Ferrante} \& {Zorzi}}]{8976281}
\bibinfo{author}{{Ciccone}, V.}, \bibinfo{author}{{Ferrante}, A.}, \&
  \bibinfo{author}{{Zorzi}, M.} (\bibinfo{year}{2020}).
\newblock \bibinfo{title}{Learning latent variable dynamic graphical models by
  confidence sets selection}.
\newblock {\it \bibinfo{journal}{IEEE Transactions on Automatic Control}\/},
  {\it \bibinfo{volume}{65}\/}, \bibinfo{pages}{5130--5143}.
\bibitem[{{De Vito} et~al.(2008){De Vito}, Massera, Piga, Martinotto \& {Di
  Francia}}]{DEVITO2008}
\bibinfo{author}{{De Vito}, S.}, \bibinfo{author}{Massera, E.},
  \bibinfo{author}{Piga, M.}, \bibinfo{author}{Martinotto, L.}, \&
  \bibinfo{author}{{Di Francia}, G.} (\bibinfo{year}{2008}).
\newblock \bibinfo{title}{On field calibration of an electronic nose for
  benzene estimation in an urban pollution monitoring scenario}.
\newblock {\it \bibinfo{journal}{Sensors and Actuators B: Chemical}\/},  {\it
  \bibinfo{volume}{129}\/}, \bibinfo{pages}{750--757}.
\bibitem[{Deistler et~al.(2010)Deistler, Anderson, Filler, Zinner \&
  Chen}]{Deistler-2010}
\bibinfo{author}{Deistler, M.}, \bibinfo{author}{Anderson, B. D.~O.},
  \bibinfo{author}{Filler, A.}, \bibinfo{author}{Zinner, C.}, \&
  \bibinfo{author}{Chen, W.} (\bibinfo{year}{2010}).
\newblock \bibinfo{title}{Generalized linear dynamic factor models: An approach
  via singular autoregressions}.
\newblock {\it \bibinfo{journal}{European Journal of Control}\/},  {\it
  \bibinfo{volume}{3}\/}, \bibinfo{pages}{211--224}.
\bibitem[{Deistler \& Zinner(2007)}]{Deistler-Z-07}
\bibinfo{author}{Deistler, M.}, \& \bibinfo{author}{Zinner, C.}
  (\bibinfo{year}{2007}).
\newblock \bibinfo{title}{Modelling high-dimensional time series by generalized
  dynamic factor models: an introductory survey}.
\newblock {\it \bibinfo{journal}{Communications on Information and Systems}\/},
   {\it \bibinfo{volume}{7}\/}, \bibinfo{pages}{153--166}.
\bibitem[{Doz et~al.(2011)Doz, Giannone \& Reichlin}]{Doz-Giannone-Reichlin-11}
\bibinfo{author}{Doz, C.}, \bibinfo{author}{Giannone, D.}, \&
  \bibinfo{author}{Reichlin, L.} (\bibinfo{year}{2011}).
\newblock \bibinfo{title}{A two-step estimator for large approximate dynamic
  factor models based on kalman filtering}.
\newblock {\it \bibinfo{journal}{Journal of Econometrics}\/},  {\it
  \bibinfo{volume}{164}\/}, \bibinfo{pages}{188--205}.
\bibitem[{Forni et~al.(2000{\natexlab{a}})Forni, Hallin, Lippi \&
  Reichlin}]{Forni-H-L-R-2000}
\bibinfo{author}{Forni, M.}, \bibinfo{author}{Hallin, M.},
  \bibinfo{author}{Lippi, M.}, \& \bibinfo{author}{Reichlin, L.}
  (\bibinfo{year}{2000}{\natexlab{a}}).
\newblock \bibinfo{title}{The generalized dynamic factor model: identification
  and estimation}.
\newblock {\it \bibinfo{journal}{The review of Economic and Statistics}\/},
  {\it \bibinfo{volume}{65}\/}, \bibinfo{pages}{453--473}.
\bibitem[{Forni et~al.(2000{\natexlab{b}})Forni, Hallin, Lippi \&
  Reichlin}]{Forni-etal-2000}
\bibinfo{author}{Forni, M.}, \bibinfo{author}{Hallin, M.},
  \bibinfo{author}{Lippi, M.}, \& \bibinfo{author}{Reichlin, L.}
  (\bibinfo{year}{2000}{\natexlab{b}}).
\newblock \bibinfo{title}{The generalized dynamic-factor model: Identification
  and estimation}.
\newblock {\it \bibinfo{journal}{The Review of Economics and Statistics}\/},
  {\it \bibinfo{volume}{82}\/}, \bibinfo{pages}{540--554}.
\bibitem[{Forni et~al.(2004)Forni, Hallin, Lippi \& Reichlin}]{Forni-H-L-R-02}
\bibinfo{author}{Forni, M.}, \bibinfo{author}{Hallin, M.},
  \bibinfo{author}{Lippi, M.}, \& \bibinfo{author}{Reichlin, L.}
  (\bibinfo{year}{2004}).
\newblock \bibinfo{title}{The generalized dynamic factor model: consistency and
  convergence rates}.
\newblock {\it \bibinfo{journal}{Journal of Econometrics}\/},  {\it
  \bibinfo{volume}{119}\/}, \bibinfo{pages}{231--235}.
\bibitem[{Forni et~al.(2005)Forni, Hallin, Lippi \& Reichlin}]{Forni-H-L-R-03}
\bibinfo{author}{Forni, M.}, \bibinfo{author}{Hallin, M.},
  \bibinfo{author}{Lippi, M.}, \& \bibinfo{author}{Reichlin, L.}
  (\bibinfo{year}{2005}).
\newblock \bibinfo{title}{The generalized dynamic factor model: One-sided
  estimation and forecasting}.
\newblock {\it \bibinfo{journal}{Journal of the American Statistical
  Association}\/},  {\it \bibinfo{volume}{100}\/}, \bibinfo{pages}{830--839}.
\bibitem[{Forni \& Lippi(2001)}]{forni_lippi_2001}
\bibinfo{author}{Forni, M.}, \& \bibinfo{author}{Lippi, M.}
  (\bibinfo{year}{2001}).
\newblock \bibinfo{title}{The generalized dynamic factor model: representation
  theory}.
\newblock {\it \bibinfo{journal}{Econometric Theory}\/},  {\it
  \bibinfo{volume}{17}\/}, \bibinfo{pages}{1113--1141}.
\bibitem[{Forni \& Reichlin(1996{\natexlab{a}})}]{Forni-Reichlin-96}
\bibinfo{author}{Forni, M.}, \& \bibinfo{author}{Reichlin, L.}
  (\bibinfo{year}{1996}{\natexlab{a}}).
\newblock \bibinfo{title}{Dynamic common factors in large cross-sections}.
\newblock {\it \bibinfo{journal}{Empirical Economics}\/},  {\it
  \bibinfo{volume}{21}\/}, \bibinfo{pages}{27--42}.
\bibitem[{Forni \& Reichlin(1996{\natexlab{b}})}]{Forni-Reichlin-98}
\bibinfo{author}{Forni, M.}, \& \bibinfo{author}{Reichlin, L.}
  (\bibinfo{year}{1996}{\natexlab{b}}).
\newblock \bibinfo{title}{Let's get real: A factor analytic approach to
  disaggregated business cycle dynamics}.
\newblock {\it \bibinfo{journal}{Review of Economic Studies}\/},  {\it
  \bibinfo{volume}{65}\/}, \bibinfo{pages}{453--473}.
\bibitem[{Geweke(1977)}]{geweke_1977}
\bibinfo{author}{Geweke, J.} (\bibinfo{year}{1977}).
\newblock \bibinfo{title}{The dynamic factor analysis of economic time series}.
\newblock In \bibinfo{editor}{D.~A. .~A. Goldberger} (Ed.), {\it
  \bibinfo{booktitle}{Latent Variables in Socio-Economic Models}\/} (pp.
  \bibinfo{pages}{365--383}).
\newblock \bibinfo{publisher}{North-Holland}.
\bibitem[{Giannone et~al.(2008)Giannone, Reichlin \& Small}]{GIANNONE2008}
\bibinfo{author}{Giannone, D.}, \bibinfo{author}{Reichlin, L.}, \&
  \bibinfo{author}{Small, D.} (\bibinfo{year}{2008}).
\newblock \bibinfo{title}{Nowcasting: The real-time informational content of
  macroeconomic data}.
\newblock {\it \bibinfo{journal}{Journal of Monetary Economics}\/},  {\it
  \bibinfo{volume}{55}\/}, \bibinfo{pages}{665--676}.
\bibitem[{Kalman(1983{\natexlab{a}})}]{KalmanDFM1983}
\bibinfo{author}{Kalman, R.} (\bibinfo{year}{1983}{\natexlab{a}}).
\newblock \bibinfo{title}{Identifability and problems of model selection in
  econometrics}.
\newblock In \bibinfo{editor}{W.~Hildebrandt} (Ed.), {\it
  \bibinfo{booktitle}{Advances in econometrics}\/}.
\newblock \bibinfo{address}{Cambridge}: \bibinfo{publisher}{Cambridge
  University Press}.
\bibitem[{Kalman(1983{\natexlab{b}})}]{Kalman-83}
\bibinfo{author}{Kalman, R.~E.} (\bibinfo{year}{1983}{\natexlab{b}}).
\newblock \bibinfo{title}{Identifiability and modeling in econometrics}.
\newblock In {\it \bibinfo{booktitle}{Developments in statistics, Vol. 4}\/}
  (pp. \bibinfo{pages}{97--136}).
\newblock \bibinfo{address}{New York}: \bibinfo{publisher}{Academic Press}
  volume~\bibinfo{volume}{4}.
\bibitem[{Kapetanios \& Marcellino(2009)}]{Kapetanios-M-09}
\bibinfo{author}{Kapetanios, G.}, \& \bibinfo{author}{Marcellino, M.}
  (\bibinfo{year}{2009}).
\newblock \bibinfo{title}{A parametric estimation method for dynamic factor
  models of large dimensions}.
\newblock {\it \bibinfo{journal}{Journal of Time Series Analysis}\/},  {\it
  \bibinfo{volume}{30}\/}, \bibinfo{pages}{208--238}.
\bibitem[{Lawley \& Maxwell(1971)}]{Lawley-M-71}
\bibinfo{author}{Lawley, D.~N.}, \& \bibinfo{author}{Maxwell, A.~E.}
  (\bibinfo{year}{1971}).
\newblock {\it \bibinfo{title}{Factor Analysis as a Statistical Method, Second
  ed.}\/}.
\newblock \bibinfo{publisher}{London: Butterworths}.
\bibitem[{Ledermann(1937)}]{Ledermann-37}
\bibinfo{author}{Ledermann, W.} (\bibinfo{year}{1937}).
\newblock \bibinfo{title}{On the rank of the reduced correlation matrix in
  multiple factor analysis}.
\newblock {\it \bibinfo{journal}{Psychometrika}\/},  {\it
  \bibinfo{volume}{2}\/}, \bibinfo{pages}{85--93}.
\bibitem[{Ledermann(1939)}]{Ledermann-39}
\bibinfo{author}{Ledermann, W.} (\bibinfo{year}{1939}).
\newblock \bibinfo{title}{On a problem concerning matrices with variable
  diagonal elements}.
\newblock {\it \bibinfo{journal}{Proc. Royal Soc. Edinburgh}\/},  {\it
  \bibinfo{volume}{XL}\/}, \bibinfo{pages}{1--17}.
\bibitem[{Lindquist \& Picci(2015)}]{LPBook}
\bibinfo{author}{Lindquist, A.}, \& \bibinfo{author}{Picci, G.}
  (\bibinfo{year}{2015}).
\newblock {\it \bibinfo{title}{Linear stochastic systems: a geometric
  approach}\/}.
\newblock \bibinfo{publisher}{Springer Verlag}.
\bibitem[{Marcellino(2017)}]{Marcellino-17}
\bibinfo{author}{Marcellino, M.} (\bibinfo{year}{2017}).
\newblock {\it \bibinfo{title}{An Introduction to Factor Modelling}\/}.
\newblock \bibinfo{type}{Lecture slides} Bocconi University.
\bibitem[{{Pe\~{n}a} \& Box(1987)}]{PenaBox1987}
\bibinfo{author}{{Pe\~{n}a}, D.}, \& \bibinfo{author}{Box, G.}
  (\bibinfo{year}{1987}).
\newblock \bibinfo{title}{Identifying a simplifying structure in time series}.
\newblock {\it \bibinfo{journal}{J. Amer. Stat. Ass.}\/},  {\it
  \bibinfo{volume}{82}\/}, \bibinfo{pages}{836--843}.
\bibitem[{{Pe\~{n}a} \& Poncela(2006)}]{PenaJSPI2006}
\bibinfo{author}{{Pe\~{n}a}, D.}, \& \bibinfo{author}{Poncela, P.}
  (\bibinfo{year}{2006}).
\newblock \bibinfo{title}{Nonstationary dynamic factor analysis}.
\newblock {\it \bibinfo{journal}{Journal of Statistical Planning and
  Inference}\/},  {\it \bibinfo{volume}{136}\/}, \bibinfo{pages}{1237--1257}.
\bibitem[{Picci \& Pinzoni(1986)}]{picci_pinzoni_86}
\bibinfo{author}{Picci, G.}, \& \bibinfo{author}{Pinzoni, S.}
  (\bibinfo{year}{1986}).
\newblock \bibinfo{title}{Dynamic factor-analysis models for stationary
  processes.}
\newblock {\it \bibinfo{journal}{IMA Journal of Math. Control and
  Information}\/},  {\it \bibinfo{volume}{3}\/}, \bibinfo{pages}{185--210}.
\bibitem[{Spearman(1904)}]{spearman-904}
\bibinfo{author}{Spearman, C.} (\bibinfo{year}{1904}).
\newblock \bibinfo{title}{General intelligence, objectively determined and
  measured}.
\newblock {\it \bibinfo{journal}{American Journal of Psychology}\/},  {\it
  \bibinfo{volume}{15}\/}, \bibinfo{pages}{201--203}.
\bibitem[{Stock \& Watson(1998)}]{Stock-W}
\bibinfo{author}{Stock, J.}, \& \bibinfo{author}{Watson, M.}
  (\bibinfo{year}{1998}).
\newblock {\it \bibinfo{title}{Diffusion Indexes}\/}.
\newblock \bibinfo{type}{Technical Report} National Bureau of Economic
  Research, Inc.
\bibitem[{Stock \& Watson(2011)}]{Stock-W-011}
\bibinfo{author}{Stock, J.}, \& \bibinfo{author}{Watson, M.}
  (\bibinfo{year}{2011}).
\newblock \bibinfo{title}{Dynamic factor models}.
\newblock In \bibinfo{editor}{M.~Clements}, \& \bibinfo{editor}{D.~Hendry}
  (Eds.), {\it \bibinfo{booktitle}{Oxford Handbook of Forecasting}\/}.
\newblock \bibinfo{publisher}{Oxford University Press}.
\bibitem[{Stock \& Watson(2015)}]{Stock-W-015}
\bibinfo{author}{Stock, J.}, \& \bibinfo{author}{Watson, M.}
  (\bibinfo{year}{2015}).
\newblock \bibinfo{title}{Factor models for macroeconomics}.
\newblock In \bibinfo{editor}{J.~B. Taylor}, \& \bibinfo{editor}{H.~Uhlig}
  (Eds.), {\it \bibinfo{booktitle}{Handbook of Macroeconomics Vol. 2}\/}.
\newblock \bibinfo{publisher}{North-Holland}.

\end{thebibliography}

\end{document}